\documentclass{amsart}

\usepackage{amsmath}
\usepackage{amssymb}
\usepackage{hyperref}
\usepackage{cases}
\usepackage{cleveref}
\usepackage{xcolor}

\newtheorem{theorem}{Theorem}[section]
\newtheorem{prop}[theorem]{Proposition}
\newtheorem{lemma}[theorem]{Lemma}
\newtheorem{cor}[theorem]{Corollary}

\theoremstyle{definition}

\theoremstyle{remark}
\newtheorem{remark}[theorem]{Remark}

\DeclareMathOperator\Div{div}

\DeclareMathOperator\tr{tr}
\def\ep{\varepsilon}    

\newcommand{\R}{\mathbb{R}}

\newcommand{\N}{\mathbb{N}}

\begin{document}

\title{Regularized $n$-Conformal heat flow and global smoothness}

\author{Woongbae Park}



\begin{abstract}
In this paper, we introduce the regularized conformal heat flow of $n$-harmonic maps, or simply regularized $n$-conformal heat flow from $n$-dimensional Riemannian manifold.
This is a system of evolution equations combined with regularized $n$-harmonic map flow and a metric evolution equation in conformal direction.
For $n=2$, the conformal heat flow does not develop finite time singularity unlike usual harmonic map flow \cite{P23} (Park, 2024).
In this paper, we show the analogous result, that regularized $n$-conformal heat flow does not develop finite time singularity unlike the (regularized) $n$-harmonic map flow.
\end{abstract}

\maketitle
\sloppy

\section{Introduction}
\label{sec1}

Let $(M,g)$ and $(N,h)$ be two Riemannian manifolds with $\dim M = n$.
The $n$-harmonic map flow is a gradient flow of $n$-energy
\begin{equation}
E(f) = E_n(f) = \frac{1}{n}\int_{M} |df|^n_g dvol_g
\end{equation}
where $e(f) = |df|^n_g$ is the energy density written in local coordinates (and with respect to $g$) by $|df|^n_g = \left(g^{ij} h_{\alpha \beta} f^\alpha_i f^\beta_j \right)^{\frac{n}{2}}$ and $dvol_g$ is the volume form of $M$ with respect to the metric $g$.
We can embed $(N,h) \hookrightarrow \R^L$ isometrically, and consider the second fundamental form $A_g(f) : T_f N \otimes T_f N \to (T_f N)^{\perp}$ (with respect to $g$).
Then we have
\begin{equation} \label{n harmonic flow}
f_t - \Delta_{n,g} f = |df|_g^{n-2} A_g(f)(df,df)
\end{equation}
where $\Delta_{n,g}$ is the $n$-Laplacian (with respect to $g$) defined by
\[
\Delta_{n,g} f = \frac{1}{\sqrt{|g|}} \frac{\partial}{\partial x^i} \left( \sqrt{|g|} |df|^{n-2}_g g^{ij} \frac{\partial}{\partial x^j} f \right) = \Div_g \left( |df|_g^{n-2} df \right)
\]
and $T_f N$ is the tangent space of $N$ at $f$.
Note that $n$-Laplacian is a degenerate elliptic operator for $n>2$.

It is well-known that for $n=2$, harmonic map flow has global weak solution, smooth everywhere except at most finitely many bubble points \cite{S85}, and such finite time singularity is explicitly constructed in \cite{CDY92}.
Moreover, if the sectional curvature of $N$ is non-positive, then the global solution is smooth \cite{ES64}.
For $n > 2$, Hungerb{\"u}hler \cite{H97} showed the following analog result.

\begin{theorem} \cite{H97}
For any $f_0 \in W^{1,n}(M,N)$, there exists a weak solution $f : M \times [0,\infty) \to N$ of \eqref{n harmonic flow} that satisfies that $\nabla f \in C^{0,\alpha}$ away from at most finitely many singular points $(x_k,T_k) \in M \times [0,\infty)$.
\end{theorem}

\begin{remark}
In \cite{H97}, Hungerb{\"u}hler stated that there exist at most finitely many singular time slices $\Sigma_k \times \{T_k\} \subset M \times [0,\infty)$ due to the finiteness of energy and energy quantization.
However, slightly more general argument can show that there are at most finitely many singular points $(x_k,T_k) \in M \times [0,\infty)$.
\end{remark}

Analogous results to $n=2$ case have also been obtained, such as a finite time singularity by Cheung-Hong \cite{CH18}, energy identity by Hong \cite{H18}, global existence under non-positive curvature condition by Fardoun-Regbaoui \cite{FR02}, global existence under small initial energy \cite{FR03}, and when domain manifold has nonempty boundary by B{\"o}gelein-Duzaar-Scheven \cite{BDS12}.
Also, more general case of $p$-harmonic map flow is studied by authors including Misawa \cite{M18a}, \cite{M18b}, \cite{M19} and Kazaniecki-\L asica-Mazowiecka-Strzelecki \cite{KLMS16}.


On the other hand, Park \cite{P22} studied a variation of harmonic map flow by combining the flow with the evolution of the metric in conformal direction, called the conformal heat flow of harmonic maps (or simply, conformal heat flow or CHF) given by
\begin{equation}
\begin{cases}
f_t =& \tau_g (f)\\
u_t =& b |df|_g^2 - a
\end{cases}
\end{equation}
where $a,b$ are positive constants with $b$ large enough, $\tau_g = \tr_g (\nabla^g df)$ is the tension field of $f$ with respect to $g$ and $g = g(x,t) = e^{2u(x,t)}g_0(x)$ is the time-dependent metric of $M$ with conformal factor $u$.
The evolution equation of the conformal factor $u$ is designed to postpone any finite time bubbling.
Thanks to the conformal invariance of the energy, the CHF enjoys many properties similar to the harmonic map flow, including $\ep$-regularity and energy decreasing.
And no finite time bubbling occurs in CHF was shown recently in \cite{P23}.

In this regard, we seek for a similar variation of $n$-harmonic map flow to overcome finite time singularity.
Since $n$-Laplacian is degenerate, we introduce the following regularized equation.
For $\ep \in [0,1]$, denote
\[
e^{\ep}(f) = (\ep |g|^{-\frac{2}{n}} + |df|_g^2)^{\frac{n}{2}}
\]
and consider the regularized $n$-energy
\begin{equation}
E^{\ep}(f) = E_n^{\ep}(f) = \frac{1}{n} \int_{M} e^{\ep}(f) dvol_g = \frac{1}{n} \int_{M} (\ep |g|^{-\frac{2}{n}} + |df|_g^2)^{\frac{n}{2}} dvol_g.
\end{equation}
Consider the following system of equations, called regularized $n$-conformal heat flow, or regularized $n$-CHF
\begin{equation} \label{eq4}
\begin{cases}
f_t =& \tau_{n,g}^{\ep}(f)\\
u_t =& b e^{\ep}(f) - a
\end{cases}
\end{equation}
where
\[
\begin{split}
\tau_{n,g}^{\ep} =& \Delta_{n,g}^{\ep}(f) + e^{\ep}(f)^{1-\frac{2}{n}} A_g(f)(df,df),\\
\Delta_{n,g}^{\ep}(f) =& \Div_{g} (e^{\ep}(f)^{1-\frac{2}{n}} df).
\end{split}
\]

This regularization does not seem the same as the version used by others.
For example, see Hungerb{\"u}hler \cite{H97}, Hong \cite{H18}, B{\"o}gelein-Duzaar-Scheven \cite{BDS12}, or Hupp-Mi\'skiewicz \cite{HM24}.
But our regularized $n$-energy becomes simpler and natural when it is written in terms of $g_0$, because
\[
E^{\ep}(f) = \frac{1}{n} \int_{M} (\ep e^{-2u} + |df|_g^2)^{\frac{n}{2}} e^{nu} dvol_0 = \frac{1}{n} \int_{M} (\ep + |df|^2)^{\frac{n}{2}} dvol_0
\]
and the equation with respect to $g_0$ is
\begin{equation} \label{eq5}
\begin{cases}
f_t =& e^{-nu} \tau_n^{\ep}(f) = e^{-nu} \left(\Delta_n^{\ep} f + (\ep + |df|^2)^{\frac{n}{2}-1} A(f)(df,df) \right)\\
u_t =& b e^{-nu} (\ep + |df|^2)^{\frac{n}{2}} - a
\end{cases}
\end{equation}
where
\[
\Delta_n^{\ep} f = \Div ((\ep + |df|^2)^{\frac{n}{2}-1} df).
\]
As a special case, we call $n$-conformal heat flow or $n$-CHF for the system of equations \eqref{eq4} (with respect to $g$) or \eqref{eq5} (with respect to $g_0$) when $\ep=0$.

Throughout this paper, we assume $(N,h) \hookrightarrow \R^L$ isometrically, and there is a constant $C_N$ such that$\|R^N\|, \|A\|, \|DA\| \leq C_N$ where $R^N$ is the Riemannian curvature tensor of $N$.
Also assume $b$ is large enough so that
\begin{equation}
C_b := \frac{nb}{2} - C_N - 2C_N^2 > 0.
\end{equation}

The following is the main theorem of this paper.
\begin{theorem} \label{main 1}
Assume $n \leq 4$ and let $f_0 \in W^{5,2}(M,N)$.
For any $\ep \in (0,1]$, there exists a smooth solution $(f,u)$ of \eqref{eq5} on $M \times [0,\infty)$ with initial condition $f(0) = f_0$, $u(0)=0$.
\end{theorem}

There are several remarks about this result.
First, the case $\ep=0$ needs more arguments due to the combined nature of $f$ and $u$.
As in \cite{H97}, $n$-harmonic map flow can be viewed as the limit case of $\ep \to 0$ and hence its existence and regularity can be obtained from degenerate parabolic system theory.
However, (regularized) $n$-CHF is not a divergence form due to the extra factor $e^{-nu}$.
Instead, we usually use the following form $e^{nu} f_t = \Delta_n^{\ep} f + (\ep + |df|^2)^{\frac{n}{2}-1} A(f)(df,df)$ which can be considered by a weighted parabolic system with density $e^{nu}(x,t)$.
Analyzing this system requires different approaches than traditional style for degenerate parabolic system in \cite{D93}.
Second, we need initial condition $f_0 \in W^{5,2}(M,N)$ which is very strong compared to usual existence theory.
This assumption is in fact technical, and one may reduce this assumption.
Third, the dimensional assumption $n \leq 4$ is used in several places including Sobolev embedding and higher order estimate $\int e^{nu}|f_t|^{n + \frac{1}{8}} \leq C$ in \Cref{sec4}.

Finally, the most important feature is that we do not have finite time singularity in this regularized $n$-CHF, unlike (regularized) $n$-harmonic map flow.
And we expect the same result holds for the limit case $\ep = 0$.
Such a result of no finite-time singularity is similar to the relations between CHF and harmonic map flow.

This paper organizes as follows.
In \Cref{sec2}, we present basic properties of regularized $n$-CHF.
Some Sobolev-type inequalities are also given.
In \Cref{sec3}, we develop local estimate for regularized $n$-CHF, and in \Cref{sec4} we further investigate regularity property, mainly obtain $\int_{B_r} e^{nu}|f_t|^{n+\frac{1}{8}} (t) < C$, see \Cref{int f_t^n}.
In \Cref{sec5}, we develop $\ep$-regularity and using it to obtain $L^q$ bound for $e_(f)$.
But the argument is not enough to derive $L^\infty$ bound due to the fact that $\|e_2(f)\|_{L^q} \to \infty$ as $q \to \infty$.
To achieve $L^\infty$ bound for $e_2(f)$, we use Moser's iteration in \Cref{sec6}.
Note that all these bounds up to \Cref{sec5} can be obtained independent of $\ep$, hence valid in particular for $n$-CHF as well.

The theory for short time existence heavily depends on uniformity of parabolic operator $\partial_t - e^{-nu} \Delta_n^{\ep}$, so we restricts $\ep>0$ in \Cref{sec7} where we show short time existence of regularized $n$-CHF using method in \cite{MM12}.
Finally, in \Cref{sec8} we show global regularity to complete \Cref{main 1}.

Throughout the paper, all the computations are obtained with respect to the metric $g_0$ unless the metric is specified.

\section{Preliminary}
\label{sec2}

In this section we recall some useful facts that are needed in the later sections.
For simplicity, we denote
\begin{equation}
e_2(f) = \ep + |df|^2.
\end{equation}
First, we provide simple observations about regularized $n$-CHF.

\begin{lemma} \label{E dec}
(Energy decreasing)
Let $(f,u)$ be a smooth solution of \eqref{eq5} on $M \times [0,T)$.
Then
\begin{equation}
\frac{d}{dt} E^{\ep}(f(t)) = -\int_{M} e^{nu} |f_t|^2 \leq 0.
\end{equation}
In particular, the energy is non-increasing and $E^{\ep}(t) \to E_\infty^{\ep}$ as $t \to \infty$.
\end{lemma}

\begin{lemma}
Let $(f,u)$ be a smooth solution of \eqref{eq5} on $M \times [0,T)$.
Then
\begin{equation} \label{e nu}
e^{nu} = e^{-nat} \left( 1 + nb \int_{0}^{t} e^{nas} e_2(f)^{\frac{n}{2}}(s) ds \right).
\end{equation}
Hence the volume $V(t) = \int_{M} dvol_g = \int_{M} e^{nu}$ satisfies
\begin{equation}
V(t) \leq e^{-nat} V(0) + \frac{nb}{a} E^{\ep}(0).
\end{equation}
\end{lemma}

Note that, from \eqref{e nu}, we can have for any $c>0$,
\[
e^{-cu} \leq e^{cat}.
\]

Next, we present various versions of Sobolev embedding.

\begin{lemma}
Let $f$ be any smooth function and $\varphi \in C_0^{\infty}(B_R)$ be a cut-off function.
Let $\beta \geq 0$.
Then we have
\begin{equation} \label{L2n}
\begin{split}
\int_{B_R} e_2(f)^{n + \beta} \varphi^n \leq& C \left( \int_{B_R} e_2(f)^{\frac{n}{2}} \right)^{\frac{2}{n}} \left( \int_{B_R} |\nabla^2 f|^2 e_2(f)^{n-2 + \beta} \varphi^n \right)\\
& + C \left( \int_{B_R} e_2(f)^{\frac{n}{2}} \right) \left( \int_{B_R} e_2(f)^{\frac{n}{2} + \beta} |\nabla \varphi|^n \right)
\end{split}
\end{equation}
for some constant $C$ only depending on $n,L$.
\end{lemma}

\begin{proof}
By H{\"o}lder inequality and Sobolev embedding $W_0^{1,1} \hookrightarrow L^{\frac{n}{n-1}}$,
\[
\begin{split}
\int_{B_R} e_2(f)^{n + \beta} \varphi^n =& \int_{B_R} e_2(f)^{\frac{1}{2}} e_2(f)^{n-\frac{1}{2} + \beta}\varphi^n\\
\leq& \left( \int_{B_R} e_2(f)^{\frac{n}{2}} \right)^{\frac{1}{n}} \left( \int_{B_R} (e_2(f)^{n-\frac{1}{2} + \beta}\varphi^n)^{\frac{n}{n-1}} \right)^{\frac{n-1}{n}}\\
\leq& C \left( \int_{B_R} e_2(f)^{\frac{n}{2}} \right)^{\frac{1}{n}} \int_{B_R} \left| \nabla (e_2(f)^{n-\frac{1}{2} + \beta}\varphi^n) \right|\\
\leq& C \left( \int_{B_R} e_2(f)^{\frac{n}{2}} \right)^{\frac{1}{n}} \int_{B_R} \left( |\nabla^2 f| |df| e_2(f)^{n-\frac{3}{2} + \beta} \varphi^n + e_2(f)^{n-\frac{1}{2} + \beta} \varphi^{n-1} |\nabla \varphi| \right).
\end{split}
\]

Now the proof is done by suitable Young's inequality.
\end{proof}

\begin{lemma} 
Let $f$ be any smooth function and $\varphi \in C_0^{\infty}(B_R)$ be a cut-off function.
Let $\beta \geq 0$.
Then we have
\begin{equation} \label{W22-r}
\begin{split}
\int_{B_R} |\nabla^2 f|^2 e_2(f)^{n-2 + \beta} \varphi^n \leq& \left( 4 + \frac{2\beta}{n-2} \right) \int |\Delta_n^{\ep} f|^2 e_2(f)^{\beta} \varphi^n + C \int e_2(f)^{n-1 + \beta} \varphi^{n-2}( \varphi^2 + |\nabla \varphi|^2 ) 
\end{split}
\end{equation}
where $C$ is a constant only depending on $n, L$ and Ricci curvature of $g_0$.

Or, we have
\begin{equation} \label{W22}
\begin{split}
\int_{B_R} |\nabla^2 f|^2 e_2(f)^{n-2 + \beta} \varphi^n \leq& \left( 4 + \frac{2\beta}{n-2} \right) \int |\Delta_n f|^2 e_2(f)^{\beta} \varphi^n + C \int e_2(f)^{n + \beta} \varphi^n\\
& + C \int e_2(f)^{\frac{n}{2} + \beta} (\varphi^n + |\nabla \varphi|^n).
\end{split}
\end{equation}
\end{lemma}

\begin{proof}
From integration by parts,
\[
\begin{split}
\int_{B_R} |\nabla^2 f|^2 e_2(f)^{n-2 + \beta} \varphi^n =& - \int_{B_R} e_2(f)^{n-2 + \beta} \langle \nabla_i \nabla_i \nabla_j f, \nabla_j f \rangle \varphi^n\\
&-2(n-2 + \beta) \int_{B_R} e_2(f)^{n-3 + \beta} \left( \langle \nabla_i df, df \rangle \right)^2 \varphi^n\\
&-n \int_{B_R} e_2(f)^{n-2 + \beta} \langle \nabla_i \nabla_j f, \nabla_j f \rangle \varphi^{n-1} \nabla_i \varphi\\
&= I + II + III.
\end{split}
\]
By Ricci identity, $I$ can be estimated by
\[
\begin{split}
I =& -\int_{B_R} e_2(f)^{n-2 + \beta} \langle \nabla_j \nabla_i \nabla_i f, \nabla_j f \rangle \varphi^n - \int_{B_R} e_2(f)^{n-2 + \beta} \langle Ric_{ij} \nabla_i f, \nabla_j f \rangle \varphi^n\\
=& I_1 + I_2.
\end{split}
\]
Note that
\[
\begin{split}
\Div \left( e_2(f)^{\frac{n}{2}-1} df \right) =& e_2(f)^{\frac{n}{2}-1} \Delta f + (n-2) e_2(f)^{\frac{n}{2}-2} \langle \nabla_j df, df \rangle \nabla_j f\\
\left| \Div \left( e_2(f)^{\frac{n}{2}-1} df  \right) \right|^2 =& e_2(f)^{n-2} |\Delta f|^2 + 2(n-2) e_2(f)^{n-3} \langle \Delta f, \nabla_j f \rangle \langle \nabla_j df, df \rangle\\
&+ (n-2)^2 e_2(f)^{n-4} \left| \langle \nabla_j df, df \rangle \nabla_j f \right|^2
\end{split}
\]
Then $I_1$ is estimated by
\[
\begin{split}
I_1 =& \int_{B_R} e_2(f)^{n-2 + \beta} |\Delta f|^2 \varphi^n + 2(n-2 + \beta) \int_{B_R} e_2(f)^{n-3 + \beta} \langle \Delta f, \nabla_j f \rangle \langle \nabla_j df, df \rangle \varphi^n\\
&+ n \int_{B_R} e_2(f)^{n-2 + \beta} \langle \Delta f, \nabla_j f \rangle \varphi^{n-1} \nabla_j \varphi\\
\leq& \frac{n-2+\beta}{n-2} \int_{B_R} \left| \Delta_n^{\ep} f \right|^2 e_2(f)^{\beta} \varphi^n + n \int_{B_R} \langle e_2(f)^{\frac{n}{2}-1} \Delta f, e_2(f)^{\frac{n}{2}-1} \nabla_j f \rangle e_2(f)^{\beta} \varphi^{n-1} \nabla_j \varphi
\end{split}
\]
where its second term is estimated by
\[
\begin{split}
n \int_{B_R}& \langle e_2(f)^{\frac{n}{2}-1} \Delta f, e_2(f)^{\frac{n}{2}-1} \nabla_j f \rangle e_2(f)^{\beta} \varphi^{n-1} \nabla_j \varphi\\
=& n \int_{B_R} \langle \Delta_n^{\ep} f, e_2(f)^{\frac{n}{2}-1} \nabla_j f \rangle e_2(f)^{\beta} \varphi^{n-1} \nabla_j \varphi\\
& - n(n-2) \int_{B_R} \langle e_2(f)^{\frac{n}{2}-2} \langle \nabla_k df, df \rangle \nabla_k f, e_2(f)^{\frac{n}{2}-1} \nabla_j f \rangle e_2(f)^{\beta} \varphi^{n-1}\nabla_j \varphi\\
\leq&  \int_{B_R} \left| \Delta_n^{\ep} f \right|^2 e_2(f)^{\beta} \varphi^n + (2n-4) \int_{B_R} e_2(f)^{n-3 + \beta} \left( \langle \nabla_k df, df \rangle \right)^2 \varphi^n\\
&+ C \int_{B_R} e_2(f)^{n-1 + \beta}  |\nabla \varphi|^2 \varphi^{n-2}.
\end{split}
\]
In conclusion, we have
\[
\begin{split}
I + II \leq& \left( 2 + \frac{\beta}{n-2} \right) \int_{B_R} |\Delta_n^{\ep} f|^2 e_2(f)^{\beta} \varphi^n + C \int_{B_R} e_2(f)^{n-1 + \beta} \varphi^n + C \int_{B_R} e_2(f)^{n-1 + \beta} |\nabla \varphi|^2 \varphi^{n-2}
\end{split}
\]
where $C$ depends on Ricci curvature of $g_0$.
Finally, note that
\[
III \leq \frac{1}{2} \int_{B_R} |\nabla^2 f|^2 e_2(f)^{n-2 + \beta} \varphi^n + C \int_{B_R} e_2(f)^{n-1 + \beta} |\nabla \varphi|^2 \varphi^{n-2}.
\]
This completes the proof of \eqref{W22-r} and using Young's inequality, we obtain \eqref{W22}.
\end{proof}

\begin{lemma} \label{L3n}
Let $f$ be any smooth function and $\varphi \in C_0^{\infty}(B_R)$ be a cut-off function.
Also assume that $\int_{B_R} e_2(f)^{\frac{n}{2}} \leq E^{\ep}(0)$.
Then we have
\[
\begin{split}
\int_{B_R} e_2(f)^{\frac{3n}{2}} \varphi^{2n} \leq& C \left( \int_{B_R} e_2(f)^\frac{n}{2} \right)^{\frac{1}{n-1}}  \left[ \left( \int_{B_R} |\nabla^2 f|^2 e_2(f)^{n-2} \varphi^n \right)^{\frac{n}{n-1}} + \left( \int_{B_R} e_2(f)^{\frac{n}{2}} |\nabla \varphi|^n \right)^{\frac{n}{n-1}} \right]
\end{split}
\]
where $C$ only depends on $n,L,E^{\ep}(0)$.
\end{lemma}

\begin{proof}
By Sobolev embedding $W^{1,p}_0 \hookrightarrow L^{q}$ with $p = \frac{2n}{n+1}$ and $q = \frac{2n}{n-1}$,
\[
\begin{split}
\int_{B_R} e_2(f)^{\frac{3n}{2}} \varphi^{2n} \leq& \int_{B_R} \left( e_2(f)|^{\frac{3n}{2q}} \varphi^{\frac{2n}{q}} \right)^q \leq C \left( \int_{B_R} \lvert \nabla e_2(f)^{\frac{3n}{2q}} \varphi^{\frac{2n}{q}} \rvert^p \right)^{\frac{q}{p}}\\
\leq& C  \left( \int_{B_R} e_2(f)^{\frac{n}{2}} \varphi^{\frac{2np}{2-p}(\frac{2}{q}-\frac{1}{2})} \right)^{\frac{1}{n-1}} \left( \int_{B_R} |\nabla^2 f|^2 e_2(f)^{n-2} \varphi^n + e_2(f)^{n-1} |\nabla \varphi|^2 \varphi^{n-2} \right)^{\frac{n}{n-1}}\\
\leq& C  \left( \int_{B_R} e_2(f)^{\frac{n}{2}} \right)^{\frac{1}{n-1}} \left( \int_{B_R} |\nabla^2 f|^2 e_2(f)^{n-2} \varphi^n + e_2(f)^{n-1} |\nabla \varphi|^2 \varphi^{n-2} \right)^{\frac{n}{n-1}}
\end{split}
\]
because $q < 4$.
On the other hand, by \eqref{L2n},
\[
\begin{split}
\int_{B_R} e_2(f)^{n-1} |\nabla \varphi|^2 \varphi^{n-2} \leq& C \int_{B_R} e_2(f)^{n} \varphi^n + C \int_{B_R} e_2(f)^{\frac{n}{2}} |\nabla \varphi|^n\\
\leq& C \left( \int_{B_R} e_2(f)^{\frac{n}{2}} \right)^{\frac{2}{n}} \left( \int_{B_R} |\nabla^2 f|^2 e_2(f)^{n-2} \varphi^n \right)\\
& + C \left( \int_{B_R} e_2(f)^{\frac{n}{2}} \right) \left( \int_{B_R} e_2(f)^{\frac{n}{2}} |\nabla \varphi|^n \right)+ C \int_{B_R} e_2(f)^{\frac{n}{2}} |\nabla \varphi|^n.
\end{split}
\]
Combine above two inequalities and using $\int_{B_R} e_2(f)^{\frac{n}{2}} \leq E^{\ep}(0)$, we get
\[
\begin{split}
\int_{B_R} e_2(f)^{\frac{3n}{2}} \varphi^{2n} \leq& C \left( \int_{B_R} e_2(f)^{\frac{n}{2}} \right)^{\frac{1}{n-1}} \left( \int_{B_R} |\nabla^2 f|^2 e_2(f)^{n-2} \varphi^n \right)^{\frac{n}{n-1}}\\
&+ C \left( \int_{B_R} e_2(f)^{\frac{n}{2}} \right)^{\frac{3}{n-1}} \left( \int_{B_R} |\nabla^2 f|^2 e_2(f)^{n-2} \varphi^n \right)^{\frac{n}{n-1}} \\
&+ C  \left( \int_{B_R} e_2(f)^{\frac{n}{2}} \right)^{\frac{1}{n-1}} \left(\int_{B_R} e_2(f)^{\frac{n}{2}} |\nabla \varphi|^n \right)^{\frac{n}{n-1}}
\end{split}
\]
which completes the proof.
\end{proof}

More general version of Sobolev embedding given below will be used.

\begin{lemma} 
Let $f$ be any smooth function and $\varphi \in C_0^{\infty}(B_R)$ be a cut-off function.
Let $\beta \geq 0$.
Then we have
\begin{equation} \label{L gamma n-r}
\left( \int_{B_R} e_2(f)^{(n-1+\beta)\frac{n}{n-2}} \varphi^{\frac{n^2}{n-2}} \right)^{\frac{n-2}{n}} \leq C  \left( \int_{B_R} |\nabla^2 f|^2 e_2(f)^{n-2 + \beta} \varphi^n + \int_{B_R} e_2(f)^{n-1+ \beta} |\nabla \varphi|^2 \varphi^{n-2} \right)
\end{equation}
for some constant $C$ only depending on $n,L$.

If we assume $\int_{B_R} e_2(f)^{\frac{n}{2}} \leq E^{\ep}(0)$, then we have
\begin{equation} \label{L gamma n}
\left( \int_{B_R} e_2(f)^{(n-1+\beta)\frac{n}{n-2}} \varphi^{\frac{n^2}{n-2}} \right)^{\frac{n-2}{n}} \leq C  \left( \int_{B_R} |\nabla^2 f|^2 e_2(f)^{n-2 + \beta} \varphi^n + \int_{B_R} e_2(f)^{\frac{n}{2} + \beta} |\nabla \varphi|^n \right)
\end{equation}
for some constant $C$ only depending on $n,L$ and $E^{\ep}(0)$.
\end{lemma}

\begin{proof}
By Sobolev embedding $W^{1,p}_0 \hookrightarrow L^q$ with $\frac{1}{p} - \frac{1}{n} = \frac{1}{q}$,
\[
\begin{split}
\int_{B_R} e_2(f)^{x} \varphi^{y} =& \int_{B_R} \left( e_2(f)^{\frac{x}{q}} \varphi^{\frac{y}{q}} \right)^q \leq C \left( \int_{B_R} \lvert \nabla e_2(f)^{\frac{x}{q}} \varphi^{\frac{y}{q}} \rvert^p \right)^{\frac{q}{p}}\\
\leq& C \left( \int_{B_R} e_2(f)^{\left(\frac{xp}{q}-\frac{p}{2}-(n-2+\beta)\frac{p}{2}\right)\frac{2}{2-p}} \varphi^{\left(\frac{yp}{q}-\frac{pn}{2}\right)\frac{2}{2-p}} \right)^{\frac{2-p}{2p}q}\\
&\cdot \left( \int_{B_R} |\nabla^2 f|^2 e_2(f)^{n-2+\beta} \varphi^n + e_2(f)^{n-1+\beta}|\nabla \varphi|^2 \varphi^{n-2} \right)^{\frac{q}{2}}
\end{split}
\]
where we want
\[
\left.
\begin{split}
\left(\frac{xp}{q}-\frac{p}{2}-(n-2+\beta)\frac{p}{2}\right)\frac{2}{2-p} =& x\\
\left(\frac{yp}{q}-\frac{pn}{2}\right)\frac{2}{2-p} =& y
\end{split}
\,\, \right\} \Rightarrow \,\,
\begin{aligned}
x =& (n-1+\beta) \frac{n}{n-2}\\
y =& \frac{n^2}{n-2}.
\end{aligned}
\]
Also note that $\frac{2-p}{2p}q = q \left( \frac{1}{q} + \frac{1}{n} - \frac{1}{2} \right) = 1 - \frac{q(n-2)}{2n}$.
Then we have
\[
\begin{split}
\left( \int_{B_R} e_2(f)^x \varphi^y \right)^{\frac{q(n-2)}{2n}} \leq& C \left( \int_{B_R} |\nabla^2 f|^2 e_2(f)^{n-2+\beta} \varphi^n + e_2(f)^{n-1+\beta}|\nabla \varphi|^2 \varphi^{n-2} \right)^{\frac{q}{2}}\\
\left( \int_{B_R} e_2(f)^x \varphi^y \right)^{\frac{n-2}{n}} \leq& C \left( \int_{B_R} |\nabla^2 f|^2 e_2(f)^{n-2+\beta} \varphi^n + e_2(f)^{n-1+\beta}|\nabla \varphi|^2 \varphi^{n-2} \right).
\end{split}
\]
This completes \eqref{L gamma n-r}.
To get \eqref{L gamma n}, observe that
\[
\int_{B_R} e_2(f)^{n-1 + \beta} |\nabla \varphi|^2 \varphi^{n-2} \leq C \int_{B_R} e_2(f)^{n + \beta} \varphi^n + C \int_{B_R} e_2(f)^{\frac{n}{2} + \beta} |\nabla \varphi|^n
\]
and the proof is done by \eqref{L2n}.
\end{proof}

\section{Local energy estimate}
\label{sec3}

In this section we work on local estimate.
Many of them are similar to CHF case, but some are changed.
We fix $B_R$ and $\varphi$ denote a suitable cut-off functions on $B_R$ and will be specified if needed.
Also, we assume there exists a constant $C_N$ only depending on the isometric embedding $(N,h) \hookrightarrow \R^L$ such that $\|R^N\|,\|A\|, \|DA\| \leq C_N$ where $R^N$ is the Riemannian curvature tensor of $N$.

We first give local estimate for $e^{npu}$.

\begin{lemma} \label{e^npu}
Let $u$ be a solution of the second equation of \eqref{eq5}.
For $p>1$ and for any $m>0$ and $t_0 \leq t$,
\begin{equation}
\int_{B_R} e^{npu}\varphi^m(t) \leq \int_{B_R} e^{npu} \varphi^m(t_0) + \frac{n(p-1)b^2}{pa} \int_{t_0}^{t} \int_{B_R} e_2(f)^{\frac{np}{2}}\varphi^m.
\end{equation}
\end{lemma}

\begin{proof}
By the second equation of \eqref{eq5},
\[
\begin{split}
\left( \int_{B_R} e^{npu} \varphi^m \right)_t =& npb \int_{B_R} e^{n(p-1)u} e_2(f)^{\frac{n}{2}} \varphi^m - npa \int_{B_R} e^{npu} \varphi^m\\
\leq& n(p-1)b \lambda \int_{B_R} e^{npu} \varphi^m + nb \lambda^{-1} \int_{B_R} e_2(f)^{\frac{np}{2}} \varphi^m - npa \int_{B_R} e^{npu} \varphi^m\\
=& nb \lambda^{-1} \int_{B_R} e_2(f)^{\frac{np}{2}} \varphi^m
\end{split}
\]
if we choose $\lambda = \frac{pa}{(p-1)b}$.
The proof is complete by integrating from $t_0$ to $t$.
\end{proof}

\begin{lemma} \label{loc E lem}
(Local energy estimate)
Let $(f,u)$ be a smooth solution of \eqref{eq5} on $B_R \times [t_1,t_2]$.
Then
\begin{equation} \label{loc E}
\begin{split}
\int_{B_R} e^{nu}|f_t|^2 \varphi^n + \frac{d}{dt} \frac{1}{n} \int_{B_R}  e_2(f)^{\frac{n}{2}} \varphi^n \leq& \delta e^{\frac{2}{n-2}nat} \int_{B_R}  e_2(f)^{\frac{n}{2}} |f_t|^2 \varphi^{n}+ C(\delta) \int_{B_R} e^{nu} |f_t|^2 |\nabla \varphi|^n\\
& + C(\delta) \int_{B_R} e_2(f)^{\frac{n}{2}} \varphi^{n} .
\end{split}
\end{equation}
\end{lemma}

\begin{proof}
For any $\delta>0$, we have
\[
\begin{split}
\frac{d}{dt} \frac{1}{n} \int_{B_R} e_2(f)^{\frac{n}{2}} \varphi^n =& \int_{B_R} e_2(f)^{\frac{n}{2}-1} \langle df ,df_t \rangle \varphi^n\\
=& -\int_{B_R} \langle \Delta_n^{\ep} f, f_t \rangle \varphi^n - n \int_{B_R} e_2(f)^{\frac{n}{2}-1} \langle df, f_t \rangle \varphi^{n-1} \nabla \varphi\\
\leq& -\int_{B_R} e^{nu} |f_t|^2 \varphi^n + \delta \int_{B_R} e_2(f)^{\frac{n}{2}-1} |f_t|^2 |\nabla \varphi|^2 \varphi^{n-2} + C(\delta) \int_{B_R} e_2(f)^{\frac{n}{2}} \varphi^{n}\\
\leq& -\int_{B_R} e^{nu}|f_t|^2 \varphi^n + \delta \int_{B_R} e^{-\frac{2}{n-2}nu} e_2(f)^{\frac{n}{2}} |f_t|^2 \varphi^{n}\\
& + C(\delta) \int_{B_R} e^{nu} |f_t|^2 |\nabla \varphi|^n + C(\delta) \int_{B_R} e_2(f)^{\frac{n}{2}} \varphi^{n}.
\end{split}
\]

\end{proof}

Note that usual local energy estimate is of the following form:
\begin{equation} \label{loc E-r}
\frac{1}{2} \int_{B_R} e^{nu}|f_t|^2 \varphi^n + \frac{d}{dt} \frac{1}{n} \int_{B_R} e_2(f)^{\frac{n}{2}} \varphi^n \leq \frac{n^2}{2} e^{nat} \int_{B_R} e_2(f)^{n-1} |\nabla \varphi|^2 \varphi^{n-2}.
\end{equation}
Unlike CHF case, the last term above cannot be estimated without additional assumption, for example, smallness of local energy.
Hungerb{\"u}hler (Theorem 1 in \cite{H97}) used this version to show that if local energy is small at $t=0$, then local energy remains small for a short time.
But in our case extra factor $e^{nu}$ make his argument inapplicable.
We will show this with different method in later section.

\begin{prop} \label{Der p=0}
(Derivative estimate for $p=0$).
Let $(f,u)$ be a smooth solution of \eqref{eq5} on $B_R \times [t_1,t_2]$.
Then
\begin{equation}
\begin{split}
\frac{d}{dt} \frac{1}{2} \int_{B_R} e^{nu}|f_t|^2 \varphi^n \leq& \frac{na}{2} \int_{B_R} e^{nu}|f_t|^2 \varphi^n\\
&+ \frac{n^3}{2} \int_{B_R} e_2(f)^{\frac{n}{2}-1} |f_t|^2 |\nabla \varphi|^2 \varphi^{n-2}\\
&-\frac{1}{4} \int_{B_R} e_2(f)^{\frac{n}{2}-1} |d f_t|^2 \varphi^n\\
&-\frac{n-2}{2} \int_{B_R} e_2(f)^{\frac{n}{2}-2} (\langle df, df_t \rangle)^2 \varphi^n\\
&+ \left( C_N + 2C_N^2 - \frac{nb}{2} \right) \int_{B_R} e_2(f)^{\frac{n}{2}} |f_t|^2 \varphi^n.
\end{split}
\end{equation}
\end{prop}

\begin{proof}
From the first equation in \eqref{eq5}, take time derivative and get
\begin{equation} \label{tau_t}
\begin{split}
(e^{nu}f_t)_t =& \Div \left( e_2(f)^{\frac{n}{2}-1} df \right)_t + \left( e_2(f)^{\frac{n}{2}-1} A(df,df) \right)_t\\
=& \Div \left( e_2(f)^{\frac{n}{2}-1} df \right)_t + (n-2)e_2(f)^{\frac{n}{2}-2} \langle df_t, df \rangle A(df,df) + 2 e_2(f)^{\frac{n}{2}-1} A(df_t,df)\\
& + e_2(f)^{\frac{n}{2}-1} DA(df,df) \cdot f_t.
\end{split}
\end{equation}
Take inner product with $f_t \varphi^n$ and integrate to get
\[
\begin{split}
\int_{B_R} \langle (e^{nu} f_t)_t, f_t \rangle  \varphi^n=& \int_{B_R} \langle \Div \left(e_2(f)^{\frac{n}{2}-1} df \right)_t, f_t  \rangle \varphi^n + 2\int_{B_R} e_2(f)^{\frac{n}{2}-1} \langle A(df_t,df), f_t \rangle \varphi^n \\
&+  \int_{B_R} e_2(f)^{\frac{n}{2}-1} \langle DA(df,df) \cdot f_t, f_t \rangle \varphi^n\\
=& -\int_{B_R} \langle \left(e_2(f)^{\frac{n}{2}-1} df \right)_t, df_t \rangle \varphi^n - \int_{B_R} \langle \left( e_2(f)^{\frac{n}{2}-1} df \right)_t, f_t \cdot \nabla \varphi \rangle n \varphi^{n-1}\\
&+ 2\int_{B_R} e_2(f)^{\frac{n}{2}-1} \langle A(df_t,df), f_t \rangle \varphi^n +  \int_{B_R} e_2(f)^{\frac{n}{2}-1} \langle DA(df,df) \cdot f_t, f_t \rangle \varphi^n.
\end{split}
\]
Here the first and second terms become
\[
\begin{split}
-\int_{B_R} \langle \left(e_2(f)^{\frac{n}{2}-1} df \right)_t, df_t \rangle \varphi^n =& -\int_{B_R} e_2(f)^{\frac{n}{2}-1} | df_t|^2 \varphi^n - (n-2) \int_{B_R} e_2(f)^{\frac{n}{2}-2} (\langle df, df_t \rangle)^2 \varphi^n\\
- \int_{B_R} \langle \left( e_2(f)^{\frac{n}{2}-1} df \right)_t, f_t \cdot \nabla \varphi \rangle n \varphi^{n-1} =& -n \int_{B_R} e_2(f)^{\frac{n}{2}-1} \langle df_t, f_t \cdot \nabla \varphi \rangle \varphi^{n-1}\\
& - n(n-2) \int_{B_R} e_2(f)^{\frac{n}{2}-2} \langle df, df_t \rangle \langle df, f_t \cdot \nabla \varphi \rangle  \varphi^{n-1}.
\end{split}
\]
On the other hand,
\[
\begin{split}
\frac{d}{dt} \frac{1}{2} \int_{B_R} e^{nu}|f_t|^2 \varphi^n =& \int_{B_R} e^{nu} \langle f_{tt}, f_t \rangle \varphi^n + \frac{n}{2} \int_{B_R} e^{nu} |f_t|^2 \varphi^n u_t.\end{split}
\]
Hence,
\[
\begin{split}
\int_{B_R} \langle (e^{nu} f_t)_t, f_t \rangle \varphi^n =& \int_{B_R} e^{nu} \langle f_{tt}, f_t \rangle \varphi^n + n \int_{B_R} e^{nu} |f_t|^2 \varphi^n u_t\\
=& \frac{d}{dt} \frac{1}{2} \int_{B_R} e^{nu} |f_t|^2 \varphi^n + \frac{n}{2} \int_{B_R} e^{nu} |f_t|^2 \varphi^n u_t\\
=& \frac{d}{dt} \frac{1}{2} \int_{B_R} e^{nu} |f_t|^2 \varphi^n + \frac{nb}{2} \int_{B_R} e_2(f)^{\frac{n}{2}} |f_t|^2 \varphi^n - \frac{na}{2} \int_{B_R} e^{nu} |f_t|^2 \varphi^n.
\end{split}
\]
Combining all together, we get
\[
\begin{split}
\frac{d}{dt} \frac{1}{2} \int_{B_R} e^{nu}|f_t|^2 \varphi^n \leq& \frac{na}{2} \int_{B_R} e^{nu}|f_t|^2 \varphi^n -\frac{nb}{2} \int_{B_R} e_2(f)^{\frac{n}{2}} |f_t|^2 \varphi^n\\
&- \int_{B_R} e_2(f)^{\frac{n}{2}-1} |df_t|^2 \varphi^n - (n-2) \int_{B_R} e_2(f)^{\frac{n}{2}-2} (\langle df, df_t \rangle )^2 \varphi^n\\
&+ n \int_{B_R} e_2(f)^{\frac{n}{2}-1} |df_t| |f_t| |\nabla \varphi| \varphi^{n-1} + n(n-2) \int_{B_R} e_2(f)^{\frac{n}{2}-\frac{3}{2}} |\langle df, df_t \rangle| |f_t| |\nabla \varphi| \varphi^{n-1}\\
&+ 2C_N \int_{B_R} e_2(f)^{\frac{n}{2}-\frac{1}{2}} |df_t| |f_t| \varphi^n + C_N \int_{B_R} e_2(f)^{\frac{n}{2}} |f_t|^2 \varphi^n.
\end{split}
\]
The remaining estimates are
\[
\begin{split}
n \int_{B_R} e_2(f)^{\frac{n}{2}-1} |df_t| |f_t| |\nabla \varphi| \varphi^{n-1} \leq& \frac{1}{4} \int_{B_R} e_2(f)^{\frac{n}{2}-1} |df_t|^2 \varphi^n + n^2 \int_{B_R} e_2(f)^{\frac{n}{2}-1} |f_t|^2 |\nabla \varphi|^2 \varphi^{n-2}\\
n(n-2) \int_{B_R} e_2(f)^{\frac{n}{2}-\frac{3}{2}} |\langle df, df_t \rangle| |f_t| |\nabla \varphi| \varphi^{n-1} \leq& \frac{1}{2}(n-2) \int_{B_R} e_2(f)^{\frac{n}{2}-2} (\langle df, df_t \rangle)^2 \varphi^n\\
& + \frac{n^2}{2} (n-2) \int_{B_R} e_2(f)^{\frac{n}{2}-1} |f_t|^2 |\nabla \varphi|^2 \varphi^{n-2}\\
2C_N \int_{B_R} e_2(f)^{\frac{n}{2}-\frac{1}{2}} |df_t| |f_t| \varphi^n \leq& \frac{1}{2} \int_{B_R} e_2(f)^{\frac{n}{2}-1} |df_t|^2 \varphi^n + 2 C_N^2 \int_{B_R} e_2(f)^{\frac{n}{2}} |f_t|^2 \varphi^n.
\end{split}
\]
This completes the proof.
\end{proof}

Form now on, we assume $b$ large enough so that
\begin{equation}
C_b := - C_N - 2C_N^2 + \frac{nb}{2} > 0.
\end{equation}

\begin{cor} \label{p=0 est}
Let $(f,u)$ be a smooth solution of \eqref{eq5} on $B_R \times [t_1,t_2]$.
Then for any $\delta>0$,
\begin{equation}
\begin{split}
\frac{1}{4} \int_{B_R} e_2(f)^{\frac{n}{2}-1} |df_t|^2 \varphi^n +& (C_b-\delta e^{\frac{2}{n-2}nat} ) \int_{B_R} e_2(f)^{\frac{n}{2}} |f_t|^2 \varphi^n +  \frac{d}{dt}\frac{1}{2} \int_{B_R} e^{nu}|f_t|^2 \varphi^n\\
 \leq & \frac{na}{2} \int_{B_R} e^{nu}|f_t|^2 \varphi^n + C(\delta) \int_{B_R} e^{nu} |f_t|^2 |\nabla \varphi|^n.
\end{split}
\end{equation}
\end{cor}

\begin{proof}
By Young's inequality, we have
\[
\begin{split}
\int_{B_R} e_2(f)^{\frac{n}{2}-1} |f_t|^2 |\nabla \varphi|^2 \varphi^{n-2} \leq& \delta \int_{B_R} e^{-\frac{2}{n-2}nu} e_2(f)^{\frac{n}{2}} |f_t|^2 \varphi^n + C(\delta) \int_{B_R} e^{nu}|f_t|^2 |\nabla \varphi|^n\\
\leq& \delta e^{\frac{2}{n-2}nat} \int_{B_R} e_2(f)^{\frac{n}{2}} |f_t|^2 \varphi^n + C(\delta) \int_{B_R} e^{nu} |f_t|^2 |\nabla \varphi|^n.
\end{split}
\]
Then by \Cref{Der p=0}, we get the conclusion.
\end{proof}

Now we can obtain estimate for $\iint e^{nu}|f_t|^2$ in terms of local energy.
Note that unlike CHF, we have H{\"o}lder-like behavior for $\iint e^{nu}|f_t|^2$.
The main difference with CHF is that for $n=2$, $n-1=\frac{n}{2}$ and hence \eqref{loc E-r} directly gives estimate for $\iint e^{nu} |f_t|^2$.
For $n>2$ however, $n-1 > \frac{n}{2}$ and we need \eqref{loc E} instead of \eqref{loc E-r} and $\iint e^{nu} |f_t|^2$ is described in terms of initial local energy.

\begin{prop} \label{iint f_t^2}
Let $(f,u)$ be a smooth solution of \eqref{eq5} on $B_R \times [t_1,t_2]$.
Then here exists $\alpha = \alpha(C_b,t_2,R) \in (0,1)$ such that for any $r \in (0,R]$,
\begin{equation}
\int_{t_1}^{t_2} \int_{B_r} e^{nu} |f_t|^2 \leq C_1 \left( \int_{B_R} e_2(f)^{\frac{n}{2}} (t_1) + \int_{B_R} e^{nu}|f_t|^2 (t_1) \right) + C_1 (t_2-t_1) E^{\ep}(0) + \left( \frac{r}{R}\right)^{\alpha} E^{\ep}(0)
\end{equation}
for some constant $C_1 = C_1(C_b,t_2)$.
\end{prop}

\begin{proof}
For simplicity, denote
\begin{equation}
K_1 := \int_{B_R} e_2(f)^{\frac{n}{2}} (t_1) + \int_{B_R} e^{nu}|f_t|^2 (t_1).
\end{equation}
Let $\varphi \in C_0^{\infty}(B_R)$ be a cut-off function such that $\varphi \equiv 1$ on $B_{\frac{R}{2}}$ with $|\nabla \varphi| \leq \frac{4}{R}$.
From \eqref{loc E} with $\delta = \frac{C_b}{2na} e^{-\frac{2}{n-2}nat_2}$, integrating over time gives
\[
\begin{split}
\int_{t_1}^{t_2} \int_{B_R} e^{nu}|f_t|^2 \varphi^n \leq& \frac{1}{n} \int_{B_R} e_2(f)^{\frac{n}{2}} \varphi^n (t_1) + \frac{C_b}{2na} \int_{t_1}^{t_2} \int_{B_R} e_2(f)^{\frac{n}{2}} |f_t|^2 \varphi^n \\
&+ C \int_{t_1}^{t_2} \int_{B_R} e^{nu}|f_t|^2 |\nabla \varphi|^n + C \int_{t_1}^{t_2} \int_{B_R} e_2(f)^{\frac{n}{2}} \varphi^n\\
\leq& \frac{1}{n} K_1 + C (t_2-t_1) E^{\ep}(0) + \frac{C_b}{2na} \int_{t_1}^{t_2} \int_{B_R} e_2(f)^{\frac{n}{2}} |f_t|^2 \varphi^n + C \int_{t_1}^{t_2} \int_{B_R} e^{nu}|f_t|^2 |\nabla \varphi|^n.
\end{split}
\]
From \Cref{p=0 est} with $\delta = \frac{C_b}{2}e^{-\frac{2}{n-2}nat_2}$, integrating over time gives
\[
\begin{split}
\frac{C_b}{2} \int_{t_1}^{t_2} \int_{B_R} e_2(f)^{\frac{n}{2}} |f_t|^2 \varphi^n \leq \frac{1}{2} \int_{B_R} e^{nu}|f_t|^2 \varphi^n (t_1) + \frac{na}{2} \int_{t_1}^{t_2} \int_{B_R} e^{nu}|f_t|^2 \varphi^n + C \int_{t_1}^{t_2} \int_{B_R} e^{nu} |f_t|^2 |\nabla \varphi|^n.
\end{split}
\]
Combining above inequalities, we get for some constant $C_1 = C_1(C_b,t_2)$,
\[
\int_{t_1}^{t_2} \int_{B_R} e^{nu}|f_t|^2 \varphi^n \leq C_1 K_1 + C_1 (t_2-t_1) E^{\ep}(0) + C_1 \int_{t_1}^{t_2} \int_{B_R} e^{nu}|f_t|^2 |\nabla \varphi|^n.
\]
This gives
\[
\begin{split}
\int_{t_1}^{t_2} \int_{B_{R/2}} e^{nu}|f_t|^2 \leq& C_1 K_1 + C_1 (t_2-t_1) E^{\ep}(0) + \frac{4^n C_1}{R^n} \int_{t_1}^{t_2} \int_{B_R \setminus B_{R/2}} e^{nu}|f_t|^2\\
\int_{t_1}^{t_2} \int_{B_{R/2}} e^{nu}|f_t|^2 \leq& \frac{C_1 K_1 + C_1 (t_2-t_1) E^{\ep}(0)}{1+\frac{4^n C_1}{R^n}}  + \theta \int_{t_1}^{t_2} \int_{B_R} e^{nu}|f_t|^2
\end{split}
\]
where
\[
\theta = \frac{\frac{4^n C_1}{R^n}}{1 + \frac{4^n C_1}{R^n}} \in (0,1).
\]
Iterating above inequality to get
\[
\int_{t_1}^{t_2} \int_{B_{R/2^k}} e^{nu}|f_t|^2 \leq \frac{C_1 K_1 + C_1 (t_2-t_1) E^{\ep}(0)}{1+\frac{4^n C_1}{R^n}}  \frac{1}{1-\theta} + \theta^k \int_{t_1}^{t_2} \int_{B_R} e^{nu}|f_t|^2.
\]
This implies that, for any $0 < r \leq R$, we have $\alpha = - \log_2(\theta) >0$ such that
\begin{equation}
\begin{split}
\int_{t_1}^{t_2} \int_{B_r} e^{nu}|f_t|^2 \leq& \frac{C_1 K_1 + C_1 (t_2-t_1) E^{\ep}(0)}{1+\frac{4^n C_1}{R^n}} \frac{1}{1-\theta} + \left( \frac{r}{R} \right)^\alpha \int_{t_1}^{t_2} \int_{B_R} e^{nu}|f_t|^2\\
=& C_1 K_1 + C_1 (t_2-t_1) E^{\ep}(0) + \left( \frac{r}{R} \right)^\alpha \int_{t_1}^{t_2} \int_{B_R} e^{nu}|f_t|^2.
\end{split}
\end{equation}

By \Cref{E dec}, we complete the proof.
\end{proof}

For simplicity, denote
\begin{equation}
K_2 := C_1 K_1 + C_1 (t_2-t_1)E^{\ep}(0) + \left( \frac{r}{R} \right)^{\alpha} E^{\ep}(0).
\end{equation}
Then \Cref{iint f_t^2} becomes $\int_{t_1}^{t_2} \int_{B_r} e^{nu}|f_t|^2 \leq K_2$.

\begin{cor} \label{n2 est}
Let $\varphi$ be a cut-off function on $B_r$ with $|\nabla \varphi| \leq \frac{4}{r}$.
Under the same assumption of \Cref{iint f_t^2}, 
\begin{equation}
\begin{split}
\int_{t_1}^{t_2} \int_{B_r} e_2(f)^{\frac{n}{2}-1} &|df_t|^2 \varphi^n ,\int_{t_1}^{t_2} \int_{B_r} e_2(f)^{\frac{n}{2}} |f_t|^2 \varphi^n \leq C_2 K_1 + C_2 K_2 \left( 1 + \frac{1}{r^{n}} \right)^2 (t_2-t_1)
\end{split}
\end{equation}
for some constant $C_2 = C_2(C_b)$.
\end{cor}

\begin{proof}
We apply \Cref{p=0 est} with $\delta = \frac{C_b}{2}e^{-\frac{2}{n-2}nat_2}$ on $B_r$ to get
\[
\begin{split}
\frac{1}{2} \int_{B_r} e^{nu}|f_t|^2 (t) \leq& \frac{1}{2} \int_{B_r} e^{nu}|f_t|^2 (t_1) + \frac{na}{2} \int_{t_1}^{t_2} \int_{B_r} e^{nu}|f_t|^2 + C\frac{1}{r^n} \int_{t_1}^{t_2} \int_{B_r} e^{nu}|f_t|^2\\
\leq& \frac{1}{2} K_1 +  \frac{na}{2} K_2 + \frac{C}{r^n} K_2.
\end{split}
\]
Then we get
\begin{equation} \label{int f_t^2}
\begin{split}
\int_{t_1}^{t_2} \int_{B_r} e^{nu}|f_t|^2 \leq& \left(  K_1 +  na K_2 + \frac{C}{r^n} K_2 \right) (t_2-t_1) \leq C K_2 \left( 1 + \frac{1}{r^n} \right) (t_2-t_1)
\end{split}
\end{equation}
hence
\[
\begin{split}
\frac{1}{4} \int_{t_1}^{t} \int_{B_r} &e_2(f)^{\frac{n}{2}-1} |df_t|^2 \varphi^n + \frac{C_b}{2} \int_{t_1}^{t} \int_{B_r} e_2(f)^{\frac{n}{2}} |f_t|^2 \varphi^n \\
\leq& \frac{1}{2} \int_{B_r} e^{nu}|f_t|^2 (t_1) + \frac{na}{2} \int_{t_1}^{t_2} \int_{B_r} e^{nu}|f_t|^2 + C\frac{1}{r^n} \int_{t_1}^{t_2} \int_{B_r} e^{nu}|f_t|^2\\
\leq& \frac{1}{2} K_1 + C K_2 \left( 1 + \frac{1}{r^{n}} \right)^2 (t_2-t_1)
\end{split}
\]
by \Cref{iint f_t^2}.
This completes the proof.
\end{proof}

Combining \Cref{loc E lem}, \Cref{iint f_t^2}, and \Cref{n2 est}, we show the following local energy estimate in terms of initial local energy.

\begin{prop} \label{loc E under small}
(Local energy estimate, finer version)
Let $(f,u)$ be a smooth solution of \eqref{eq5} on $B_R \times [t_1,t_2]$.
Then for any $t \in [t_1,t_2]$ and for any $r \in (0,\frac{R}{2}]$,
\begin{equation}
\begin{split}
\int_{B_r} e_2(f)^{\frac{n}{2}} (t) \leq& \int_{B_{2r}} e_2(f)^{\frac{n}{2}} (t_1) +  C_2 \left( \int_{B_R} e_2(f)^{\frac{n}{2}} (t_1) + \int_{B_R} e^{nu}|f_t|^2 (t_1) \right)\\
&  + C_3  \left( 1 + \frac{1}{(2r)^n} \right)^2 (t_2-t_1)
\end{split}
\end{equation}
for some constant $C_3 = C_3(R,C_b,t_2)$.
\end{prop}

\begin{proof}
Note that $K_2 \leq (C_1 + C_1 t_2 + 1) E^{\ep}(0)$.
From \Cref{n2 est} and \eqref{int f_t^2} on $B_{2r}$, integrate over $[t_1,t]$ with $t \in [t_1,t_2]$,
\[
\begin{split}
\int_{t_1}^{t_2} \int_{B_{2r}} e_2(f)^{\frac{n}{2}} |f_t|^2 \leq& C_2 K_1 + C_2 (C_1 + C_1 t_2 + 1) E^{\ep}(0) \left( 1 + \frac{1}{(2r)^n} \right)^2 (t_2-t_1)\\
\int_{t_1}^{t_2} \int_{B_{2r}} e^{nu}|f_t|^2 |\nabla \varphi|^n \leq& C (C_1 + C_1 t_2 + 1) E^{\ep}(0)\left(1 + \frac{1}{(2r)^n} \right) \frac{t_2-t_1}{(2r)^n}\\
\leq& C (C_1 + C_1 t_2 + 1) E^{\ep}(0) \left( 1 + \frac{1}{(2r)^n} \right)^2 (t_2-t_1) .
\end{split}
\]
Hence, by \eqref{loc E}, for any $t \in [t_1,t_2]$,
\[
\begin{split}
 \int_{B_{2r}} e_2(f)^{\frac{n}{2}} \varphi^n (t) \leq&  \int_{B_{2r}} e_2(f)^{\frac{n}{2}} \varphi^n (t_1) +  n e^{\frac{2}{n-2} naT} \int_{t_1}^{t} \int_{B_{2r}} e_2(f)^{\frac{n}{2}} |f_t|^2 \varphi^n\\
& + C \int_{t_1}^{t} \int_{B_{2r}} e^{nu}|f_t|^2 |\nabla \varphi|^n + C \int_{t_1}^{t} \int_{B_{2r}} e_2(f)^{\frac{n}{2}} \varphi^n\\
\int_{B_r} e_2(f)^{\frac{n}{2}} (t) \leq& \int_{B_{2r}} e_2(f)^{\frac{n}{2}} (t_1) +  C_2 K_1 + C_3  \left( 1 + \frac{1}{(2r)^n} \right)^2 (t_2-t_1)
\end{split}
\]
for some constant $C_3 = C_3(C_b,t_2,E^{\ep}(0))$.
\end{proof}

\section{Higher order estimate}
\label{sec4}

In this section we get higher order estimate $\int e^{nu}|f_t|^{p+2} \leq C$.
Ultimately we will get above inequality for $p +2 = n + \frac{1}{8}$.
And this inequality can be obtained under the condition $p \leq \frac{4n-8+4 \frac{1}{8}}{(n-2)^2}$, which is true when $p \leq 2 + \frac{1}{8}$ because we only consider $n \leq 4$.

Remark that the number $\frac{1}{8}$ above can be improved by some constant $c(b) < 1$ where $c(b)$ depends on $b$.
Or, we can fix $c<1$ first and obtain $\int e^{nu} |f_t|^{n+c} \leq C$ and require $b=b(c)$ large enough in terms of $c$.
But in the next section we will see that $\int e^{nu} |f_t|^{n+\frac{1}{8}} \leq C$ is enough.

Throughout this section, $\varphi$ is a cut-off function on $B_r$ with $|\nabla \varphi| \leq \frac{4}{r}$.
First we get the following estimate.

\begin{lemma} \label{p+2 est}
Let $(f,u)$ be a smooth solution of \eqref{eq5} on $B_R \times [t_1,t_2]$.
For $p \geq 1$ and $r \in (0,R]$,
\begin{equation}
\begin{split}
\int_{t_1}^{t_2} \int_{B_r} e^{nu}|f_t|^{p+2} \varphi^n \leq& C \int_{t_1}^{t_2} \int_{B_r} e_2(f)^{\frac{n}{2}-1} |f_t|^{p-1}  |f_{t i}|^2 \varphi^n + C \int_{t_1}^{t_2} \int_{B_r} e^{nu} |f_t|^{p+1} |\nabla \varphi|^n\\
& + C' \int_{t_1}^{t_2} \int_{B_r} e_2(f)^{\frac{n}{2}} |f_t|^{p+1} \varphi^n
\end{split}
\end{equation}
where $C$ only depends on $p$ and $C'$ only depends on $p,t_2$.
\end{lemma}

\begin{proof}
From the first equation in \eqref{eq5}, taking inner product with $e^{nu} f_t$ gives
\[
e^{nu} |f_t|^2 = \langle \Div \left( e_2(f)^{\frac{n}{2}-1}df \right), f_t \rangle = - e_2(f)^{\frac{n}{2}-1} \langle f_i, f_{ti} \rangle + \nabla_i ( e_2(f)^{\frac{n}{2}-1} \langle f_i, f_t \rangle ).
\]
Multiplying with $|f_t|^p \varphi^n$ for $p \geq 1$ and integrating gives
\[
\begin{split}
\int_{B_r} e^{nu} |f_t|^{p+2} \varphi^n =& - \int_{B_r} e_2(f)^{\frac{n}{2}-1} |f_t|^{p} \langle f_i, f_{ti} \rangle \varphi^n - p \int_{B_r} e_2(f)^{\frac{n}{2}-1} |f_t|^{p-2} \langle f_{ti}, f_t \rangle \langle f_i, f_t \rangle \varphi^n\\
& - n \int_{B_r} e_2(f)^{\frac{n}{2}-1} |f_t|^p \langle f_i, f_t \rangle \varphi^{n-1} \nabla_i \varphi\\
=& I + II + III. 
\end{split}
\]
Each term can be estimate by
\[
\begin{split}
I,II \leq& C \int_{B_r} e_2(f)^{\frac{n}{2}-1} |f_t|^{p-1} |f_{ti}|^2 \varphi^n + C \int_{B_r} e_2(f)^{\frac{n}{2}} |f_t|^{p+1} \varphi^n\\
III \leq& C \int_{B_r} e^{nu}|f_t|^{p+1} |\nabla \varphi|^n  + C \int_{B_r} e^{-\frac{n}{n-1}u} e_2(f)^{\frac{n}{2}} |f_t|^{p+1} \varphi^n\\
\leq& C \int_{B_r} e^{nu}|f_t|^{p+1} |\nabla \varphi|^n + C e^{\frac{n}{n-1}at} \int_{B_r} e_2(f)^{\frac{n}{2}} |f_t|^{p+1} \varphi^n.
\end{split}
\]
This completes the proof.
\end{proof}

\begin{cor} \label{iint f_t^3}
Under the same assumption of \Cref{iint f_t^2}, for any $r \in (0,R]$,
\begin{equation}
\int_{t_1}^{t_2} \int_{B_r} e^{nu} |f_t|^3 \varphi^n \leq C_4
\end{equation}
for some constant $C_4 = C_4(\ep_1,r,R,E^{\ep}(0),C_b,t_1,t_2)$.
\end{cor}

\begin{proof}
From \Cref{p+2 est} with $p=1$ on $B_r$, \Cref{iint f_t^2} and \Cref{n2 est},
\[
\begin{split}
\int_{t_1}^{t_2} \int_{B_r} e^{nu} |f_t|^3 \varphi^n \leq& C \int_{t_1}^{t_2} \int_{B_r} e_2(f)^{\frac{n}{2}-1}|f_{t i}|^2 \varphi^n + C \frac{1}{r^n}\int_{t_1}^{t_2} \int_{B_r} e^{nu}|f_t|^2\\
&+ C' \int_{t_1}^{t_2} \int_{B_r} e_2(f)^{\frac{n}{2}} |f_t|^2 \varphi^n\\
\leq& C_4
\end{split}
\]
for some constant $C_4 = C_4(\ep_1,r,R,E^{\ep}(0),C_b,t_1,t_2)$.
\end{proof}

Next, we derive $p$-version of \Cref{Der p=0}.

\begin{prop} \label{Der p}
(Derivative estimate for $p$).
Assume $n \leq 4$ and $p \leq 2 + \frac{1}{8}$.
Let $(f,u)$ be a smooth solution of \eqref{eq5} on $B_R \times [t_1,t_2]$.
Then for any $r \in (0,R]$,
\begin{equation}
\begin{split}
\frac{d}{dt} \frac{1}{p+2} \int_{B_r} e^{nu}|f_t|^{p+2} \varphi^n \leq& \frac{(p+1)na}{p+2} \int_{B_r} e^{nu}|f_t|^{p+2} \varphi^n\\
&+ n^2(n-1)^2 \int_{B_r} e_2(f)^{\frac{n}{2}-1} |f_t|^{p+2} |\nabla \varphi|^2 \varphi^{n-2}\\
&-\frac{1}{8} \int_{B_r} e_2(f)^{\frac{n}{2}-1} |f_t|^p |d f_t|^2 \varphi^n\\
&+ \left( C_N + 2C_N^2 - \frac{(p+1)nb}{p+2} \right) \int_{B_r} e_2(f)^{\frac{n}{2}} |f_t|^{p+2} \varphi^n.
\end{split}
\end{equation}
\end{prop}

\begin{proof}
Like proof of \Cref{Der p=0}, we take inner product with $f_t |f_t|^p \varphi^n$ to \eqref{tau_t}.
\[
\begin{split}
\int_{B_r} \langle (e^{nu} f_t)_t,& f_t |f_t|^p \rangle  \varphi^n\\
=& \int_{B_r} \langle \Div \left(e_2(f)^{\frac{n}{2}-1} df \right)_t, f_t |f_t|^p  \rangle \varphi^n + 2 \int_{B_r} e_2(f)^{\frac{n}{2}-1} \langle A(df_t,df), f_t |f_t|^p \rangle \varphi^n \\
&+  \int_{B_r} e_2(f)^{\frac{n}{2}-1} \langle DA(df,df) \cdot f_t, f_t |f_t|^p \rangle \varphi^n\\
=& -\int_{B_r} \langle \left(e_2(f)^{\frac{n}{2}-1} df \right)_t, df_t |f_t|^p \rangle \varphi^n - \int_{B_r} \langle \left( e_2(f)^{\frac{n}{2}-1} df \right)_t, f_t |f_t|^p \cdot \nabla \varphi \rangle n \varphi^{n-1}\\
&- p \int_{B_r} \langle \left( e_2(f)^{\frac{n}{2}-1} df \right)_t, f_t |f_t|^{p-2} \langle df_t, f_t \rangle \rangle \varphi^n\\
&+ 2 \int_{B_r} e_2(f)^{\frac{n}{2}-1} \langle A(df_t,df), f_t |f_t|^p \rangle \varphi^n +  \int_{B_r} e_2(f)^{\frac{n}{2}-1} \langle DA(df,df) \cdot f_t, f_t |f_t|^p \rangle \varphi^n.
\end{split}
\]
And
\[
\begin{split}
-\int_{B_r} \langle \left(e_2(f)^{\frac{n}{2}-1} df \right)_t, df_t |f_t|^p \rangle \varphi^n =& -\int_{B_r} e_2(f)^{\frac{n}{2}-1} |f_t|^p |df_t|^2 \varphi^n\\
& - (n-2) \int_{B_r} e_2(f)^{\frac{n}{2}-2} |f_t|^p (\langle df, df_t \rangle)^2 \varphi^n\\
\end{split}
\]
\[
\begin{split}
- \int_{B_r} \langle \left( e_2(f)^{\frac{n}{2}-1} df \right)_t,& f_t |f_t|^p \cdot \nabla \varphi \rangle n \varphi^{n-1} \\
=& -n \int_{B_r} e_2(f)^{\frac{n}{2}-1} |f_t|^p \langle df_t, f_t \cdot \nabla \varphi \rangle \varphi^{n-1}\\
& - n(n-2) \int_{B_r} e_2(f)^{\frac{n}{2}-2} |f_t|^p \langle df, df_t \rangle \langle df, f_t \cdot \nabla \varphi \rangle  \varphi^{n-1}\\
- p \int_{B_r} \langle \left( e_2(f)^{\frac{n}{2}-1} df \right)_t, &f_t |df|^{p-2} \langle df_t, f_t \rangle \rangle \varphi^n\\
 =& -p \int_{B_r} e_2(f)^{\frac{n}{2}-1} |f_t|^{p-2} ( \langle df_t, f_t \rangle)^2 \varphi^n\\
&- p(n-2) \int_{B_r} e_2(f)^{\frac{n}{2}-2} |f_t|^{p-2} \langle df_t, df \rangle \langle df, f_t \rangle \langle df_t, f_t \rangle \varphi^n.
\end{split}
\]
On the other hand,
\[
\begin{split}
\int_{B_r} \langle (e^{nu} f_t)_t, f_t |f_t|^p \rangle \varphi^n =& \int_{B_r} e^{nu} |f_t|^p \langle f_{tt}, f_t \rangle \varphi^n + n \int_{B_r} e^{nu} |f_t|^{p+2} \varphi^n u_t\\
=& \frac{d}{dt} \frac{1}{p+2} \int_{B_r} e^{nu} |f_t|^{p+2} \varphi^n + \frac{(p+1)n}{p+2} \int_{B_r} e^{nu} |f_t|^{p+2} \varphi^n u_t\\
=& \frac{d}{dt} \frac{1}{p+2} \int_{B_r} e^{nu} |f_t|^{p+2} \varphi^n + \frac{(p+1)nb}{p+2} \int_{B_r} e_2(f)^{\frac{n}{2}} |f_t|^{p+2} \varphi^n\\
& - \frac{(p+1)na}{p+2} \int_{B_r} e^{nu} |f_t|^{p+2} \varphi^n.
\end{split}
\]
Combining all together, we get
\[
\begin{split}
\frac{d}{dt} \frac{1}{p+2} \int_{B_r} e^{nu}|f_t|^{p+2} \varphi^n \leq& \frac{(p+1)na}{p+2} \int_{B_r} e^{nu}|f_t|^{p+2} \varphi^n -\frac{(p+1)nb}{p+2} \int_{B_r} e_2(f)^{\frac{n}{2}} |f_t|^{p+2} \varphi^n\\
&- \int_{B_r} e_2(f)^{\frac{n}{2}-1} |f_t|^p |df_t|^2 \varphi^n - (n-2) \int_{B_r} e_2(f)^{\frac{n}{2}-2} |f_t|^p (\langle df, df_t \rangle )^2 \varphi^n\\
&+ n \int_{B_r} e_2(f)^{\frac{n}{2}-1} |f_t|^{p+1} |df_t|  |\nabla \varphi| \varphi^{n-1}\\
& + n(n-2) \int_{B_r} e_2(f)^{\frac{n}{2}-\frac{3}{2}} |f_t|^{p+1} |\langle df, df_t \rangle|  |\nabla \varphi| \varphi^{n-1}\\
& -p \int_{B_r} e_2(f)^{\frac{n}{2}-1} |f_t|^{p-2} ( \langle df_t, f_t \rangle)^2 \varphi^n\\
&+ p(n-2) \int_{B_r} e_2(f)^{\frac{n}{2}-\frac{3}{2}} |f_t|^{p-1} |\langle df_t, df \rangle| |\langle df_t, f_t \rangle| \varphi^n \\
&+ 2C_N \int_{B_r} e_2(f)^{\frac{n}{2}-\frac{1}{2}} |f_t|^{p+1} |df_t|  \varphi^n + C_N \int_{B_r} e_2(f)^{\frac{n}{2}} |f_t|^{p+2} \varphi^n.
\end{split}
\]

To estimate the term with $p(n-2)$, we need other terms and additional assumption for $n \leq 4$ and $p \leq 2 + \frac{1}{8}$.
Denote $A^2 = e_2(f)^{\frac{n}{2}-2}|f_t|^p (\langle df, df_t \rangle)^2$ and $B^2 = e_2(f)^{\frac{n}{2}-1} |f_t|^{p-2} (\langle df_t, f_t \rangle)^2$ with $A,B>0$.
From the observation $A^2 \leq e_2(f)^{\frac{n}{2}-1} |f_t|^p |df_t|^2$, we have
\[
-\frac{1}{8} \int_{B_r} e_2(f)^{\frac{n}{2}-1}|f_t|^p |df_t|^2 \varphi^n - (n-2) \int_{B_r} A^2 \varphi^n \leq -(n-2+\frac{1}{8}) \int_{B_r} A^2 \varphi^n.
\]
Then
\[
-(n-2 + \frac{1}{8}) \int_{B_r} A^2 \varphi^n -p \int_{B_r} B^2 \varphi^n + p(n-2) \int_{B_r} AB \varphi^n \leq 0
\]
for any $A,B>0$ if and only if the inequality $(n-2 + \frac{1}{8})x^2 - p(n-2)x + p \geq 0$ holds for all $x>0$.
The minimum of the polynomial occurs at $x=\frac{p(n-2)}{2n-4+\frac{1}{4}}>0$ and corresponding minimum value is $-\frac{p^2 (n-2)^2}{4n-8+\frac{1}{2}} + p \geq 0$, hence we get
\[
p \leq \frac{4n-8+\frac{1}{2}}{(n-2)^2}.
\]
And this is true under our assumption $n \leq 4$ and $p \leq 2 + \frac{1}{8}$.

Now terms without $p$ can be estimated by
\[
\begin{split}
n \int_{B_r} e_2(f)^{\frac{n}{2}-1} &|f_t|^{p+1} |df_t|  |\nabla \varphi| \varphi^{n-1} + n(n-2) \int_{B_r} e_2(f)^{\frac{n}{2}-\frac{3}{2}} |f_t|^{p+1} |\langle df, df_t \rangle| |\nabla \varphi| \varphi^{n-1}\\
\leq& n(n-1) \int_{B_r} e_2(f)^{\frac{n}{2}-1} |f_t|^{p+1} |df_t| |\nabla \varphi| \varphi^{n-1}\\
\leq& \frac{1}{4} \int_{B_r} e_2(f)^{\frac{n}{2}-1} |f_t|^p |df_t|^2 \varphi^n + n^2(n-1)^2 \int_{B_r} e_2(f)^{\frac{n}{2}-1} |f_t|^{p+2} |\nabla \varphi|^2 \varphi^{n-2}\\
2C_N \int_{B_r} e_2(f)^{\frac{n}{2}-\frac{1}{2}}  &|f_t|^{p+1} |df_t| \varphi^n\\
\leq& \frac{1}{2} \int_{B_r} e_2(f)^{\frac{n}{2}-1} |f_t|^p |df_t|^2 \varphi^n + 2 C_N^2 \int_{B_r} e_2(f)^{\frac{n}{2}} |f_t|^{p+2} \varphi^n.
\end{split}
\]
This completes the proof.
\end{proof}

Note that our assumption of $b$ satisfies
\[
- C_N - 2C_N^2 + \frac{(p+1)nb}{p+2} > C_b > 0.
\]

\begin{cor} \label{p est}
Let $(f,u)$ be a smooth solution of \eqref{eq5} on $B_R \times [t_1,t_2]$.
Assume $n \leq 4$ and $p \leq 2 + \frac{1}{8}$.
Then for any $\delta>0$ and for any $r \in (0,R]$,
\begin{equation}
\begin{split}
\frac{1}{4} \int_{B_r} e_2(f)^{\frac{n}{2}-1} |f_t|^p |df_t|^2 \varphi^n +& (C_b-\delta e^{\frac{2}{n-2}nat} ) \int_{B_r} e_2(f)^{\frac{n}{2}} |f_t|^{p+2} \varphi^n +  \frac{d}{dt}\frac{1}{p+2} \int_{B_r} e^{nu}|f_t|^{p+2} \varphi^n\\
 \leq & \frac{(p+1)na}{p+2} \int_{B_r} e^{nu}|f_t|^{p+2} \varphi^n + C(\delta) \int_{B_r} e^{nu} |f_t|^{p+2} |\nabla \varphi|^n.
\end{split}
\end{equation}
\end{cor}

\begin{proof}
Proof is similar with proof of \Cref{p=0 est}.
\end{proof}

We will now derive regularity for $\int e^{nu}|f_t|^3$ from that of $\iint e^{nu}|f_t|^3$.
Note that \Cref{p est} implies the following Gronwall-type inequality
\begin{equation}
\frac{d}{dt} \int_{B_r} e^{nu}|f_t|^{p+2} \varphi^n \leq (p+1)na \int_{B_r} e^{nu}|f_t|^{p+2} + C_5 \frac{1}{r^n} \int_{B_r} e^{nu}|f_t|^{p+2}
\end{equation}
for some constant $C_5 = C_5(C_b,t_2)$ which implies that for any $t,t_0 \in [t_1,t_2]$ with $t_0 \leq t$,
\begin{equation} \label{Gronwall}
\int_{B_r} e^{nu}|f_t|^{p+2} \varphi^n (t) \leq e^{(p+1)na(t-t_0)} \left( \int_{B_r} e^{nu}|f_t|^{p+2} \varphi^n (t_0) + C_5 \frac{1}{r^n} \int_{t_0}^{t} \int_{B_r} e^{nu}|f_t|^{p+2} \right).
\end{equation}
Hence, using \Cref{iint f_t^3}, we can obtain the following result.

\begin{lemma}\label{int f_t^3}
Under the same assumption of \Cref{iint f_t^2}, for any $t \in [t_{1,\lambda},t_2]$ and for any $r \in (0,R]$,
\begin{equation}
\int_{B_r} e^{nu} |f_t|^3 \varphi^n (t) \leq C_6 
\end{equation}
for some constant $C_6 = C_6(\ep_1,r,R,E^{\ep}(0),C_b,t_1,t_2, \lambda)$.
Here $t_{1,\lambda} = (1-\lambda) t_1+ \lambda t_2$ for some $\lambda \in (0,1)$.
\end{lemma}

\begin{proof}
By \Cref{iint f_t^3}, and since $\varphi$ is arbitrary,
\[
\int_{t_1}^{t_2} \int_{B_r} e^{nu}|f_t|^3 \leq C_4.
\]
Fix $t \in [t_{1,\lambda},t_2]$ and take $t_0 \in [t_1,t]$ such that
\[
\int_{B_r} e^{nu}|f_t|^3 \varphi^n (t_0) = \min_{t_1 \leq s \leq t} \int_{B_r} e^{nu}|f_t|^3 \varphi^n (s).
\]
Then by \Cref{Gronwall} with $p=1$,
\[
\begin{split}
\int_{B_r} e^{nu}|f_t|^{3} \varphi^n (t) \leq& e^{2na(t-t_0)} \left( \int_{B_r} e^{nu}|f_t|^{3} \varphi^n (t_0) + C_5 \frac{1}{r^n} \int_{t_0}^{t} \int_{B_r} e^{nu}|f_t|^{3} \right)\\
\leq& e^{2na(t-t_0)} \left( \frac{1}{t-t_1} \int_{t_1}^{t} \int_{B_r} e^{nu}|f_t|^3 \varphi^n + C_5 \frac{1}{r^n} \int_{t_0}^{t} \int_{B_r} e^{nu}|f_t|^3 \right)\\
\leq& e^{2na(t_2-t_1)} \left( \frac{1}{\lambda(t_2-t_1)} + C_5 \frac{1}{r^n} \right) C_4 = C_6 
\end{split}
\]
for some constant $C_6 = C_6(\ep_1,r,R,E^{\ep}(0),C_b,t_1,t_2,\lambda)$.
\end{proof}

\begin{cor} \label{n3 est}
Under the same assumption of \Cref{iint f_t^2}, for any $r \in (0,R]$,
\begin{equation}
\int_{t_{1,\lambda}}^{t_2} \int_{B_r} e_2(f)^{\frac{n}{2}-1} |f_t| |df_t|^2 \varphi^n ,\quad \int_{t_{1,\lambda}}^{t_2} \int_{B_r} e_2(f)^{\frac{n}{2}} |f_t|^3 \varphi^n \leq C_7 
\end{equation}
for some constant $C_7 = C_7(\ep_1,r,R,E^{\ep}(0),C_b,t_1,t_2, \lambda)$.
\end{cor}

\begin{proof}
We apply \Cref{p est} with $p=1$, $\delta = \frac{C_b}{2}e^{-\frac{2}{n-2}nat_2}$ on $B_r$ and using \Cref{iint f_t^3} and \Cref{int f_t^3}, integrate over $[t_{1,\lambda},t_2]$ to get
\[
\begin{split}
\frac{1}{4} \int_{t_{1,\lambda}}^{t_2} \int_{B_r} &e_2(f)^{\frac{n}{2}-1} |f_t| |df_t|^2 \varphi^n + \frac{C_b}{2} \int_{t_{1,\lambda}}^{t_2} \int_{B_r} e_2(f)^{\frac{n}{2}} |f_t|^3 \varphi^n\\
\leq& \frac{1}{3} \int_{B_r} e^{nu}|f_t|^3 (t_{1,\lambda}) + \frac{2na}{3} \int_{t_{1,\lambda}}^{t_2} \int_{B_r} e^{nu}|f_t|^3 + C\frac{1}{r^n} \int_{t_{1,\lambda}}^{t_2} \int_{B_r} e^{nu}|f_t|^3\\
\leq& C_7
\end{split}
\]
for some constant $C_7 = C_7(\ep_1,r,R,E^{\ep}(0),C_b,t_1,t_2, \lambda)$.
\end{proof}

Now we manipulate our regularity machinery.
First, from \Cref{p+2 est} with $p=2$ on $B_r$, \Cref{iint f_t^3} and \Cref{n3 est},
\[
\begin{split}
\int_{t_{1,\lambda}}^{t_2} \int_{B_r} e^{nu} |f_t|^4 \varphi^n \leq& C \int_{t_{1,\lambda}}^{t_2} \int_{B_r} e_2(f)^{\frac{n}{2}-1} |f_t| |f_{t i}|^2 \varphi^n + C \frac{1}{r^n}\int_{t_{1,\lambda}}^{t_2} \int_{B_r} e^{nu}|f_t|^3\\
&+ C_3 \int_{t_{1,\lambda}}^{t_2} \int_{B_r} e_2(f)^{\frac{n}{2}} |f_t|^3 \varphi^n\\
\leq& C_8
\end{split}
\]
for some constant $C_8 = C_8(\ep_1,r,R,E^{\ep}(0),C_b,t_1,t_2, \lambda)$.

Then, similar to the argument in \Cref{int f_t^3}, for any $t \in [t_{1,\lambda^2},t_2]$ with $t_{1,\lambda^2} = (1-\lambda^2)t_1 + \lambda^2 t_2$, we get
\begin{equation} \label{int f_t^4}
\int_{B_r} e^{nu} |f_t|^4 \varphi^n (t) \leq C_9 
\end{equation}
for some constant $C_9 = C_9(\ep_1,r,R,E^{\ep}(0),C_b,t_1,t_2,\lambda)$.

And similar argument in \Cref{n3 est}, we get
\begin{equation} \label{n4 est}
\int_{t_{1,\lambda^2}}^{t_2} \int_{B_r} e_2(f)^{\frac{n}{2}-1} |f_t|^2 |df_t|^2 \varphi^n ,\quad \int_{t_{1,\lambda^2}}^{t_2} \int_{B_r} e_2(f)^{\frac{n}{2}} |f_t|^4 \varphi^n \leq C_{10}
\end{equation}
for some constant $C_{10} = C_{10}(\ep_1,r,R,E^{\ep}(0),C_b,t_1,t_2,\lambda)$.

Keep working on, by \Cref{p+2 est} with $p=3$ on $B_r$ and above estimates, we get
\begin{equation}\label{iint f_t^5}
\int_{t_{1,\lambda^2}}^{t_2} \int_{B_r} e^{nu}|f_t|^5 \varphi^n \leq C_{11} 
\end{equation}
for some constant $C_{11} = C_{11}(\ep_1,r,R,E^{\ep}(0),C_b,t_1,t_2,\lambda)$.

Note that the bootstrap argument above cannot get the estimate like
\[
\int_{B_r} e^{nu}|f_t|^5 \varphi^n (t) \leq C
\]
because \eqref{Gronwall} does not hold for $p = 3$.
However, because \eqref{Gronwall} holds for $p=2+\frac{1}{8}$, and by inerpolation, with suitable choice of $\lambda$, we have that for any $t \in [t_1',t_2]$ with $t_1' = \frac{t_1+t_2}{2}$,
\begin{equation}
\int_{B_r} e^{nu}|f_t|^{4+\frac{1}{8}} \varphi^n (t) \leq C
\end{equation}
for some constant $C$.
More precisely, as we assume $n \leq 4$, we have the following lemma.

\begin{lemma}\label{int f_t^n}
Under the same assumption of \Cref{iint f_t^2}, for any $t \in [t_1',t_2]$ and for any $r \in (0,R]$,
\begin{equation}
\int_{B_r} e^{nu} |f_t|^{n} \varphi^n (t),  \int_{B_r} e^{nu} |f_t|^{n + \frac{1}{8}} \varphi^n (t)  \leq C_{12}
\end{equation}
for some constant $C_{12} = C_{12}(\ep_1,r,R,E^{\ep}(0),C_b,t_1,t_2)$.
Here $t_1' = \frac{t_1+t_2}{2}$.
\end{lemma}

\section{$L^q$ estimate}
\label{sec5}

In this section we obtain $L^q$ estimate for $e_2(f)$.
We first show the $\ep$ regularity.
Briefly, we are going to show that if $\sup_t \int_{B_r} e_2(f)^{\frac{n}{2}} (t) \leq \ep$ (and more assumptions), then we have a good control of higher derivative of $f$, in our case, $\int |\nabla^2 f|^2 e_2(f)^{n-2} (t)$ for all $t \in [t_1',t_2]$.
And combined with Sobolev embedding, this will give a control of $\int e_2(f)^{\frac{3n}{2}} (t)$.
Such an additional regularity can improve regularity for higher derivative, say $\int |\nabla^2 f|^2 e_2(f)^{n-2+\beta} (t)$ for all $t \in [t_1',t_2]$ for $\beta \leq n$, which now gives a control of $\int e_2(f)^{\frac{7n}{2}}(t)$.
We can continue this argument to obtain $\int e_2(f)^q \leq C$ for any $q>0$, but because the bound depends on $\beta$, it may blow up as $\beta \to \infty$.
In the next section, instead of using such bootstrap-style argument, we adopt Moser's iteration to get $L^\infty$ bounds for $e_2(f)$.

The following theorem is the $\ep$ regularity for $\int |\nabla^2 f|^2 e_2(f)^{n-2}$.

\begin{theorem}\label{W22+0}
There exists $\ep_1>0$ that only depends on $n,L, C_N$ and Ricci curvature of $g_0$ such that the following holds:

Let $(f,u)$ be a smooth solution of \eqref{eq5} on $B_r \times [t_1,t_2]$.
Also assume that
\[
\sup_{t \in [t_1,t_2]} \int_{B_r} e_2(f)^{\frac{n}{2}} (t) \leq \ep_1, \quad \int_{B_r} e^{nu}|f_t|^2 (t_1) \leq \ep_1, \quad \int_{B_r} e^{3nu} (t_1') \leq \ep_1
\]
where $t_1' = \frac{t_1+t_2}{2}$.
Then for any $t \in [t_1',t_2]$,
\begin{equation}
\int_{B_r} |\nabla^2 f|^2 e_2(f)^{n-2} (t) \leq C_{13} = C_{13}(\ep_1,r,E^{\ep}(0),C_b,t_1,t_2,Ric).
\end{equation}
\end{theorem}

\begin{proof}
From the equation $\Delta_n^{\ep} f + e_2(f)^{\frac{n}{2}-1}A(df,df) = e^{nu}f_t$, we have
\[
\int_{B_r} |\Delta_n^{\ep} f|^2 \varphi^n \leq C_N^2 \int_{B_r} e_2(f)^{n} \varphi^n + \int_{B_r} e^{2nu}|f_t|^2 \varphi^n.
\]
Combined with \eqref{W22} and \eqref{L2n} for $\beta=0$, we have
\[
\begin{split}
\int_{B_r} |\nabla^2 f|^2 e_2(f)^{n-2} \varphi^n \leq& 4 \int_{B_r} |\Delta_n^{\ep} f|^2 \varphi^n + C \int_{B_r} e_2(f)^{n} \varphi^n + C \left( 1 + \frac{1}{r^n} \right) \ep_1\\
\leq& C(1 + C_N^2) \int_{B_r} e_2(f)^{n} \varphi^n + C \left( 1 + \frac{1}{r^n} \right) \ep_1 + C\int_{B_r} e^{2nu}|f_t|^2 \varphi^n\\
\leq& C(1+C_N^2) \ep_1^{\frac{2}{n}} \left( \int_{B_r} |\nabla^2 f|^2 e_2(f)^{n-2} \varphi^n \right) + C(1+C_N^2) \ep_1 \frac{\ep_1}{r^n}\\
&+ C \left( 1 + \frac{1}{r^n} \right) \ep_1 + C\int_{B_r} e^{2nu}|f_t|^2 \varphi^n
\end{split}
\]
where the constant $C$ only depends on $n,L$ and Ricci curvature of $g_0$.
Now if $\ep_1$ is small enough so that $C(1+C_N^2)\ep_1^{\frac{2}{n}} \leq \frac{1}{2}$, we can have
\[
\int_{B_r} |\nabla^2 f|^2 e_2(f)^{n-2} \varphi^n \leq C \left( 1 + \frac{1}{r^n} \right)\ep_1 + C \int_{B_r} e^{2nu}|f_t|^2 \varphi^n.
\]

Next, note that by \eqref{int f_t^4}
\[
\begin{split}
\int_{B_r} e^{2nu}|f_t|^2 \varphi^n \leq& \left( \int_{B_r} e^{nu}|f_t|^4 \right)^{\frac{1}{2}} \left( \int_{B_r} e^{3nu} \varphi^{2n} \right)^{\frac{1}{2}} \leq C_9^{\frac{1}{2}}  \left( \int_{B_r} e^{3nu} \varphi^{2n} \right)^{\frac{1}{2}}
\end{split}
\]
where by \Cref{e^npu} and \Cref{L3n}, the last factor can be estimated by
\[
\begin{split}
\int_{B_r} e^{3nu}\varphi^{2n} \leq& \int_{B_r} e^{3nu} \varphi^{2n}(t_1') + C \int_{t_1'}^{t} \int_{B_r} e_2(f)^{\frac{3n}{2}} \varphi^{2n}\\
\leq&  \ep_1 + C \int_{t_1'}^{t} \ep_1^{\frac{1}{n-1}} \left( \left( \int_{B_r} |\nabla^2 f|^2 e_2(f)^{n-2} \varphi^n \right)^{\frac{n}{n-1}} + \left( \frac{1}{r^n} \ep_1\right)^{\frac{n}{n-1}} \right)\\
\leq& \ep_1 + C (t-t_1') \ep_1^{\frac{n+1}{n-1}} \left(\frac{1}{r^{n}}\right)^{\frac{n}{n-1}} + C \ep_1^{\frac{1}{n-1}} \int_{t_1'}^{t} \left( \int_{B_r} |\nabla^2 f|^2 e_2(f)^{n-2} \varphi^n \right)^{\frac{n}{n-1}}\\
\leq& C \ep_1 \left( 1 + \frac{1}{r^n} \right)^{\frac{n}{n-1}} + C \ep_1^{\frac{1}{n-1}}\int_{t_1'}^{t} \left( \int_{B_r} |\nabla^2 f|^2 e_2(f)^{n-2} \varphi^n \right)^{\frac{n}{n-1}}
\end{split}
\]
where constant $C$ depends on $E^{\ep}(0),t_1,t_2$.
Combine all together, we have that for $X(t) = \left( \int_{B_r} |\nabla^2 f|^2 e_2(f)^{n-2} \varphi^n (t) \right)^{\frac{n}{n-1}}$,
\[
\begin{split}
X(t) \leq& C \left( 1 + \frac{1}{r^{n}}\right)^{\frac{n}{n-1}} \ep_1^{\frac{n}{n-1}} + C C_9^{\frac{n}{2(n-1)}} \left( C \ep_1 \left( 1 + \frac{1}{r^n} \right)^{\frac{n}{n-1}} + C \ep_1^{\frac{1}{n-1}} \int_{t_1'}^{t} X(s)ds \right)^{\frac{1}{2}}
\end{split}
\]
or
\begin{equation}
\begin{split}
X(t) \leq& \sqrt{D_1 + D_2 \int_{t_1'}^{t} X(s)ds}
\end{split}
\end{equation}
where $D_1,D_2>0$ are constants depending on $\ep_1,r,E^{\ep}(0),C_b,t_1,t_2$ and Ricci curvature of $g_0$.
Solving this inequality gives
\[
\sqrt{D_1 + D_2 \int_{t_1'}^{t} X(s)ds} \leq \frac{D_2}{2} (t-t_1') + \sqrt{D_1}
\]
hence we obtain
\[
X(t) \leq \frac{D_2}{2} (t-t_1') + \sqrt{D_1} \leq C_{13}
\]
for some constant $C_{13} = C_{13}(\ep_1,r,E^{\ep}(0),C_b,t_1,t_2,Ric)$.
This completes the proof.
\end{proof}

\begin{cor} \label{L3n bdd}
Under the same assumption of \Cref{W22+0},
for any $t \in [t_1',t_2]$,
\begin{equation}
\int_{B_r} e_2(f)^{\frac{3n}{2}}(t) \leq C_{14} = C_{14} (\ep_1,r,E^{\ep}(0),C_b,t_1,t_2,Ric).
\end{equation}
\end{cor}

\begin{proof}
The conclusion is from \Cref{W22+0} and \Cref{L3n}.
\end{proof}

The next theorem is the bootstrap-style regularity theorem.
The assumption $\beta \leq n$ is needed because we use the bound $\int e_2(f)^{\frac{n}{2}+\beta} \leq C$ in the proof.

\begin{theorem}\label{W22+beta}
There exists $\ep_2>0$ that only depends on $n,L, C_N$ and Ricci curvature of $g_0$ such that the following holds:

Let $(f,u)$ be a smooth solution of \eqref{eq5} on $B_r \times [t_1,t_2]$ and $0 \leq \beta \leq n$.
Also assume that
\begin{equation} \label{assume}
\sup_{t \in [t_1,t_2]} \int_{B_r} e_2(f)^{\frac{n}{2}} (t) \leq \ep_2, \quad \int_{B_r} e^{nu}|f_t|^2 (t_1) \leq \ep_2, \quad \int_{B_r} e^{\frac{2(n-1+\beta)}{n-2}nu} (t_1') \leq C(\beta)
\end{equation}
where $t_1' = \frac{t_1+t_2}{2}$ and $C(\beta)$ only depends on $\beta$.
Then 
for any $t \in [t_1',t_2]$,
\begin{equation}
\int_{B_r} e_2(f)^{(n-1 + \beta)\frac{n}{n-2}} (t) \leq C_{15} = C_{15}(\ep_2,r,E^{\ep}(0),C_b,t_1,t_2,Ric,\beta).
\end{equation}
\end{theorem}

\begin{proof}
The proof works similar to that of \Cref{W22+0}.
From the equation $e^{nu}f_t = \Delta_n^{\ep} f + e_2(f)^{\frac{n}{2}-1}A(df,df)$, we have
\[
\int_{B_r} |\Delta_n^{\ep} f|^2 e_2(f)^{\beta} \varphi^n \leq C_N^2 \int_{B_r} e_2(f)^{n+\beta} \varphi^n + \int_{B_r} e^{2nu}|f_t|^2 e_2(f)^{\beta}\varphi^n.
\]
Combined with \eqref{W22} and \eqref{L2n}, and from $\beta \leq n$ and \Cref{L3n bdd}, we have
\[
\begin{split}
\int_{B_r} |\nabla^2 f|^2 e_2(f)^{n-2 + \beta} \varphi^n \leq& \left( 4 + \frac{2\beta}{n-2} \right) \int_{B_r} |\Delta_n^{\ep} f|^2 e_2(f)^{\beta} \varphi^n + C \int_{B_r} e_2(f)^{n + \beta} \varphi^n + C \left( 1 + \frac{1}{r^n} \right) C_{14}\\
\leq& C'(1 + C_N^2) \int_{B_r} e_2(f)^{n + \beta} \varphi^n + C \left( 1 + \frac{1}{r^n} \right) C_{14}\\
& + C\int_{B_r} e^{2nu}|f_t|^2 e_2(f)^{\beta} \varphi^n\\
\leq& C'(1+C_N^2) \ep_2^{\frac{2}{n}} \left( \int_{B_r} |\nabla^2 f|^2 e_2(f)^{n-2 + \beta} \varphi^n \right) + C(1+C_N^2) \ep_1 \frac{C_{14}}{r^n}\\
&+ C \left( 1 + \frac{1}{r^n} \right) C_{14} + C\int_{B_r} e^{2nu}|f_t|^2 e_2(f)^{\beta} \varphi^n
\end{split}
\]
where the constant $C'$ only depends on $n,L$ and Ricci curvature of $g_0$.
Now if $\ep_2$ is small enough so that $C'(1+C_N^2)\ep_2^{\frac{2}{n}} \leq \frac{1}{2}$, we can have
\[
\int_{B_r} |\nabla^2 f|^2 e_2(f)^{n-2 + \beta} \varphi^n \leq C \left( 1 + \frac{1}{r^n} \right)C_{14} + C \int_{B_r} e^{2nu}|f_t|^2 e_2(f)^{\beta} \varphi^n.
\]

Now, by H{\"o}lder, we have
\[
\begin{split}
\int_{B_r} e^{2nu}|f_t|^2 e_2(f)^{\beta} \varphi^n \leq& \left( \int_{B_r} e^{nu}|f_t|^n \right)^{\frac{2}{n}} \left( \int_{B_r} e_2(f)^{(n-1+\beta)\frac{n}{n-2}} \varphi^{\frac{n^2}{n-2}} \right)^{ \frac{\beta(n-2)}{(n-1+\beta)n}} \\
& \qquad \cdot \left( \int_{B_r} e^{\gamma nu} \varphi^{n \frac{n-1}{n-1+\beta} \frac{n \gamma}{2n-2}} \right)^{\frac{2n-2}{n \gamma}}
\end{split}
\]
where
\[
\frac{2}{n} + \frac{\beta(n-2)}{(n-1+\beta)n} + \frac{2n-2}{n \gamma} = 1 \Longrightarrow \,\, \gamma = \frac{2(n-1+\beta)}{n-2}.
\]
Note that $\frac{\gamma n}{2} = (n-1+\beta)\frac{n}{n-2}$ and $n \frac{n-1}{n-1+\beta} \frac{n \gamma}{2n-2} = \frac{n^2}{n-2}$.
Then we have
\[
\begin{split}
\int_{B_r} e^{2(n-1+\beta)\frac{n}{n-2} u}  \varphi^{ \frac{n^2}{n-2}}  \leq& \int_{B_r} e^{2(n-1+\beta)\frac{n}{n-2} u}  \varphi^{ \frac{n^2}{n-2}} (t_1') + C \int_{t_1'}^{t} \int_{B_r} e_2(f)^{(n-1+\beta)\frac{n}{n-2}}  \varphi^{ \frac{n^2}{n-2}} \\
\leq& C(\beta) + C \int_{t_1'}^{t} \int_{B_r} e_2(f)^{(n-1+\beta)\frac{n}{n-2}}  \varphi^{ \frac{n^2}{n-2}} .
\end{split}
\]
Together with \eqref{L gamma n} and \Cref{int f_t^n}, we have
\[
\begin{split}
&\left( \int_{B_r} e_2(f)^{(n-1+\beta)\frac{n}{n-2}} \varphi^{\frac{n^2}{n-2}} \right)^{\frac{n-2}{n}}\\
 \leq& C \int_{B_r} |\nabla^2 f|^2 e_2(f)^{n-2+\beta} \varphi^n + C_{14} \frac{1}{r^n}\\
\leq& C \left( 1 + \frac{1}{r^n} \right)C_{14} + C C_{12}^{\frac{2}{n}}  \left( \int_{B_r} e_2(f)^{(n-1+\beta)\frac{n}{n-2}} \varphi^{\frac{n^2}{n-2}} \right)^{\frac{\beta}{n-1+\beta} \frac{n-2}{n}}\\
&\qquad \qquad \qquad \qquad \qquad \cdot \left( C(\beta) + C \int_{t_1'}^{t} \int_{B_r} e_2(f)^{(n-1+\beta)\frac{n}{n-2}}  \varphi^{\frac{n^2}{n-2}} \right)^{\frac{n-1}{n-1+\beta} \frac{n-2}{n}}\\
\leq&  C \left( 1 + \frac{1}{r^n} \right) C_{14} + \frac{1}{2} \left(\int_{B_r} e_2(f)^{(n-1+\beta)\frac{n}{n-2}} \varphi^{\frac{n^2}{n-2}} \right)^{\frac{n-2}{n}}\\
& + C C_{12}^{\frac{2(n-1+\beta)}{n(n-1)}} \left( C(\beta) + C \int_{t_1'}^{t} \int_{B_r} e_2(f)^{(n-1+\beta)\frac{n}{n-2}} \varphi^{\frac{n^2}{n-2}} \right)^{\frac{n-2}{n}}.
\end{split}
\]

So, if we let $X(t) = \int_{B_r}e_2(f)^{(n-1+\beta) \frac{n}{n-2}} \varphi^{\frac{n^2}{n-2}}$, we have
\[
X(t) \leq D_1 + D_2 \int_{t_1'}^{t} X(s)ds
\]
for some constant $D_1,D_2>0$ depending on $\ep_2,r,E^{\ep}(0),C_b,t_1,t_2$ and Ricci curvature of $g_0$ and $\beta$.
By Gronwall's inequality, we have
\[
X(t) \leq D_1 (1 + D_2 (t-t_1') e^{D_2 (t-t_1')}) = C_{15}
\]
for some constant $C_{15} = C_{15}(\ep_2,r,E^{\ep}(0),C_b,t_1,t_2,Ric,\beta)$.

This completes the proof.
\end{proof}

From $\beta \leq n$ and by applying above theorem, we can get an improved bound $\int e_2(f)^{\frac{n(2n-1)}{n-2}} \leq C$, which allows that the assumption for $\beta$ can be extended to $\beta \leq 3n$.
In other words, we get the bound
\begin{equation} \label{Lq}
\int_{B_r} e_2(f)^{q}(t) \leq C_{15}
\end{equation}
for any $q \leq \frac{7n}{2}$ and for any $t \in [t_1',t_2]$.

And by applying the theorem, we can get more improved bound $\int e_2(f)^{\frac{n(4n-1)}{n-2}} \leq C$ and repeatedly apply the argument to get $\int e_2(f)^q \leq C$ for any $q > 0$.

There are several disadvantages in this recursive argument.
First, the assumption for $u$ and its bounds depend on $\beta$, hence we cannot simply take $\beta \to \infty$.
Actually we may expect $C(\beta)$ in \eqref{assume} will diverge as $\beta \to \infty$.
Second, the coefficient $4 + \frac{2\beta}{n-2}$ in \eqref{W22} cannot have uniform bound independent on $\beta$ if there is no pre-determined upper bound for $\beta$.
This observation suggests that the constant $C_{15}$ above will blow up as $\beta \to \infty$ because $4 + \frac{2\beta}{n-2} \to \infty$ as $\beta \to \infty$.
So, using \Cref{W22+beta} repeatedly could not lead to $L^\infty$ bound for $e_2(f)$.

\section{$L^\infty_{loc}$ estimate}
\label{sec6}

In this section we use Moser's iteration argument to obtain $L^\infty_{loc}$ estimate for $e_2(f)$.
Denote
\begin{equation} \label{ep}
\bar{\ep} = \min \{\ep_1, \ep_2\}
\end{equation}
where $\ep_1$ is in \Cref{W22+0} and $\ep_2$ is in \Cref{W22+beta}.

Assume that
\begin{equation} \label{ass}
\sup_{t \in [t_1,t_2]} \int_{B_r} e_2(f)^{\frac{n}{2}} (t) \leq \bar{\ep}, \quad \int_{B_r} e^{nu}|f_t|^2 (t_1) \leq \bar{\ep}, \quad  \sup_{x \in B_r} u(x, t_1') \leq C_{u}
\end{equation}
where $t_1' = \frac{t_1+t_2}{2}$ and $C_{u}$ is a constant only depending on $u(t_1')$.
Then by \Cref{int f_t^n}, for any $t \in [t_1',t_2]$,
\[
\int_{B_r} e^{nu} |f_t|^{n+\frac{1}{8}} \varphi^n (t) \leq C_{12} = C_{12}( \bar{\ep},r,E^{\ep}(0),C_b,t_1,t_2).
\]

\begin{theorem} \label{L infty}
Let $(f,u)$ be a smooth solution of \eqref{eq5} on $B_r \times [t_1,t_2]$ and $\bar{\ep}$ be as in \eqref{ep}.
Also assume \eqref{ass} holds.
Then
\begin{equation}
\sup_{B_{\frac{r}{2}} \times [t_1',t_2]} e_2(f) \leq C_{19}
\end{equation}
for some constant $C_{19}$ only depending on $\bar{\ep},r,E^{\ep}(0),C_b,t_1,t_2,Ric,C_u,C_N$. 
\end{theorem}

Therefore, we obtain $\|df\|_{L^\infty(B_{\frac{r}{2}} \times [t_1',t_2])} \leq \sqrt{C_{19}} < \infty$.

\begin{proof}
Consider $0 < \rho' < \rho \leq r $ and a cut-off function $\varphi \in C^\infty_0(B_\rho)$ such that
\[
\varphi \equiv 1 \text{ on } B_{\rho'}, \quad |\nabla \varphi| \leq \frac{C}{\rho-\rho'}.
\]

From the equation $e^{nu}f_t = \Delta_n^{\ep} f + e_2(f)^{\frac{n}{2}-1}A(df,df)$, we have that for any $\beta \geq 0$,
\[
\int_{B_\rho} |\Delta_n^{\ep} f|^2 e_2(f)^{\beta} \varphi^n \leq C_N^2 \int_{B_\rho} e_2(f)^{n+\beta} \varphi^n + \int_{B_\rho} e^{2nu}|f_t|^2 e_2(f)^{\beta}\varphi^n.
\]
Combined with \eqref{W22-r} and \eqref{L gamma n-r}, we have for any $t \in [t_1',t_2]$,
\[
\begin{split}
&\left( \int_{B_\rho} e_2(f)^{(n-1+\beta)\frac{n}{n-2}} \varphi^{\frac{n^2}{n-2}} \right)^{\frac{n-2}{n}}\\
 \leq& C \left(\int_{B_\rho} |\nabla^2 f|^2 e_2(f)^{n-2 + \beta} \varphi^n  + \int_{B_\rho} e_2(f)^{n-1+\beta} |\nabla \varphi|^2 \varphi^{n-2} \right)\\
\leq& C \left( 4 + \frac{2\beta}{n-2} \right) \int_{B_\rho} |\Delta_n^{\ep} f|^2 e_2(f)^{\beta} \varphi^n  + C \int_{B_\rho} e_2(f)^{n-1+\beta} \varphi^{n-2} (\varphi^2 + |\nabla \varphi|^2)\\
\leq&  C\left( 4 + \frac{2\beta}{n-2} \right)C_N^2  \int_{B_\rho} e_2(f)^{n + \beta} \varphi^n + C \int_{B_\rho} e_2(f)^{n-1+\beta} \varphi^{n-2} (\varphi^2 + |\nabla \varphi|^2)\\
& + C \left( 4 + \frac{2\beta}{n-2} \right) \int_{B_\rho} e^{2nu}|f_t|^2 e_2(f)^{\beta} \varphi^n
\end{split}
\]
where the constants $C$ only depends on $n,L$ and Ricci curvature of $g_0$.

To estimate the last term, by H{\"o}lder, we have
\[
\begin{split}
\int_{B_\rho} e^{2nu}|f_t|^2 e_2(f)^{\beta} \varphi^n \leq& \left( \int_{B_\rho} e^{nu}|f_t|^{n+\frac{1}{8}} \right)^{\frac{2}{n+\frac{1}{8}}} \left( \int_{B_\rho} e_2(f)^{\mu}\right)^{ \frac{\beta}{\mu}} \left( \int_{B_\rho} e^{ 2\mu u} \right)^{1-\frac{2}{n+\frac{1}{8}} - \frac{\beta}{\mu}}
\end{split}
\]
where $\mu$ satisfies
\begin{equation}
2n = n \frac{2}{n+\frac{1}{8}} + 2\mu \left( 1-\frac{2}{n+\frac{1}{8}} - \frac{\beta}{\mu} \right) \Longrightarrow  \mu = \frac{n+\frac{1}{8}} {n+\frac{1}{8}-2} \left(n \frac{n+\frac{1}{8}-1}{n+\frac{1}{8}} + \beta \right).
\end{equation}
In particular,
\[
n + \beta < \mu < \frac{n}{n-2} (n-1+\beta).
\]
Also, by \Cref{e^npu}, we have
\[
\begin{split}
\int_{B_\rho} e^{2 \mu u}  \leq& \int_{B_\rho} e^{2 \mu u} (t_1') + C \int_{t_1'}^{t} \int_{B_\rho} e_2(f)^{\mu} \leq e^{2\mu C_{u}} + C \int_{t_1'}^{t} \int_{B_\rho} e_2(f)^{\mu}\\
\leq& e^{2\mu C_{u}} + C (t-t_1') \sup_{[t_1',t_2]} \int_{B_\rho} e_2(f)^{\mu}.
\end{split}
\]
Combining together, we have
\[
\begin{split}
\int_{B_\rho} e^{2nu}|f_t|^2 e_2(f)^\beta \varphi^n \leq& C \left( \sup_{[t_1',t_2]} \int_{B_\rho} e_2(f)^{\mu} \right)^{\frac{\beta}{\mu}} \left( e^{2\mu C_u} + C (t-t_1') \sup_{[t_1',t_2]} \int_{B_\rho} e_2(f)^{\mu} \right)^{\frac{n \frac{n+\frac{1}{8}-1}{n+\frac{1}{8}}}{\mu}}\\
\leq& C \left( \sup_{[t_1',t_2]}   \left( \int_{B_\rho} e_2(f)^{\mu} \right)^{\frac{n+\frac{1}{8} -2}{n+\frac{1}{8} }} + 1 \right)
\end{split}
\]
where $C$ depends on $r,\bar{\ep},C_b,t_1,t_2, C_u$ and is independent on $\mu$, $\beta$ and $\rho$.

In conclusion, we get
\begin{equation}
\begin{split}
\left( \sup_{[t_1',t_2]} \int_{B_{\rho'}} e_2(f)^{(n-1+\beta)\frac{n}{n-2}}  \right)^{\frac{n-2}{n}} \leq& C \left( 4 + \frac{2\beta}{n-2} \right) C_N^2 \left( \sup_{[t_1',t_2]} \int_{B_\rho} e_2(f)^{\mu} \right)^{\frac{n+\beta}{\mu}}\\
& + C \left( 1 + \frac{1}{(\rho-\rho')^2} \right) \left( \sup_{[t_1',t_2]} \int_{B_\rho} e_2(f)^{\mu} \right)^{\frac{n-1+\beta}{\mu}}\\
& + C  \left( 4 + \frac{2\beta}{n-2} \right) \left(  \sup_{[t_1',t_2]}  \left( \int_{B_\rho} e_2(f)^{\mu} \right)^{\frac{n+\frac{1}{8} -2}{n+\frac{1}{8} }} + 1 \right).
\end{split}
\end{equation}

Now we define sequences of $\beta_k,\mu_k,\rho_k$ such that $\beta_0 = 0$ and that
\begin{equation} \label{Moser iter}
\begin{split}
\mu_{k+1} =& \frac{n+\frac{1}{8}} {n+\frac{1}{8}-2} \left(n \frac{n+\frac{1}{8}-1}{n+\frac{1}{8}} + \beta_{k+1} \right) = \frac{n}{n-2} (n-1+\beta_k)\\
\rho_k =& r \left( \frac{1}{2} + \frac{n}{4(\frac{n}{2}+\beta_k)} \right).
\end{split}
\end{equation}
Denote
\begin{equation}
\lambda = \frac{n+\frac{1}{8}-2}{n+\frac{1}{8}} \frac{n}{n-2}>1.
\end{equation}
Then we have
\[
\beta_{k+1} + n \frac{n+\frac{1}{8}-1}{n+\frac{1}{8}} = \lambda (\beta_k + n-1) \Longrightarrow \beta_{k+1} + \frac{n}{2} = \lambda(\beta_k + \frac{n}{2})
\]
which gives $\beta_k = \frac{n}{2} \left( \lambda^k-1 \right) \to \infty$ as $k \to \infty$, hence $\rho_k \to \frac{r}{2}$ as $k \to \infty$.
Also, for $\rho=\rho_k, \rho'=\rho_{k+1}$,
\[
\begin{split}
\rho-\rho' =& r \left( \frac{n}{4(\frac{n}{2}+\beta_k)}-\frac{n}{4(\frac{n}{2}+\beta_{k+1})} \right)\\
=&r \frac{n}{4} \frac{\beta_{k+1}-\beta_k}{(\frac{n}{2}+\beta_k)(\frac{n}{2}+\beta_{k+1})}\\
=&r \frac{n}{4} \frac{\lambda-1}{\lambda} \frac{1}{\frac{n}{2}+\beta_k}
\end{split}
\]
which implies $ \left( 1 + \frac{1}{(r-r')^2} \right) \leq C (1 + \beta_k)^2$ for some constant $C$ only depending on $r,\lambda$.
Finally, observe that
\[
\frac{n+\frac{1}{8} -2}{n+\frac{1}{8} } = \frac{n \frac{n+\frac{1}{8}-1}{n+\frac{1}{8}}+\beta_k}{\mu_k}< \frac{n+\beta_k}{\mu_k} < 1.
\]
Then using $(a^p+b^p) \leq (a+b)^p \leq 2^{p-1}(a^p+b^p)$ for $0<p<1$, \eqref{Moser iter} becomes
\[
\begin{split}
\left( \sup_{[t_1',t_2]} \int_{B_{\rho_{k+1}}} e_2(f)^{(n-1+\beta_k)\frac{n}{n-2}}  \right)^{\frac{n-2}{n}} \leq& C(1 +  \beta_k)^2 \left( \sup_{[t_1',t_2]} \left( \int_{B_{\rho_k}} e_2(f)^{\mu_k}  \right)^{ \frac{n+\beta_k}{\mu_k}} + 1 \right)\\
\left( \sup_{[t_1',t_2]} \int_{B_{\rho_{k+1}}} e_2(f)^{(n-1+\beta_k)\frac{n}{n-2}} +1  \right)^{\frac{n-2}{n}} \leq&C_{16}(1 +  \beta_k)^2 \left( \sup_{[t_1',t_2]} \int_{B_{\rho_k}} e_2(f)^{\mu_k}   + 1 \right)^{ \frac{n+\beta_k}{\mu_k}}
\end{split}
\]
where $C_{16}$ only depends on $r,\bar{\ep},C_b,t_1,t_2,C_u,\lambda$, and $C_N$ and is independent on $k$.
Denote 
\begin{equation}
M_k :=   \left( \sup_{[t_1',t_2]} \int_{B_{\rho_k}} e_2(f)^{\mu_k} + 1 \right)^{\frac{1}{\mu_k}}.
\end{equation}
Then above inequality becomes
\[
M_{k+1} \leq C_{16}^{\frac{1}{n-1+\beta_k}} (1+\beta_k)^{\frac{2}{n-1+\beta_k}} M_k^{\frac{n+\beta_k}{n-1+\beta_k}}
\]
and by iteration, we obtain
\begin{equation} \label{Moser result}
\begin{split}
M_k \leq& C_{16}^{\frac{1}{n-1+\beta_k} + \frac{n+\beta_k}{n-1+\beta_k}\frac{1}{n-1+\beta_{k-1}} + \cdots + \left( \prod_{j=1}^{k} \frac{n+\beta_j}{n-1+\beta_j} \right) \frac{1}{n-1}}\\
& \cdot (1+\beta_k)^{\frac{2}{n-1+\beta_k}} (1+\beta_{k-1})^{\frac{n+\beta_k}{n-1+\beta_k}\frac{2}{n-1+\beta_{k-1}}} \cdots (1+\beta_1)^{ \left( \prod_{j=2}^{k}\frac{n+\beta_j}{n-1+\beta_j} \right) \frac{2}{n-1+\beta_1}}\\
&\cdot M_0^{\prod_{j=0}^{k} \frac{n+\beta_j}{n-1+\beta_j}}.
\end{split}
\end{equation}

At this stage we need the following claims.

{\underline {Claim 1}:}
\[
\prod_{j=0}^{k} \frac{n+\beta_j}{n-1+\beta_j} \leq C_{17}
\]
where $C_{17}$ is independent on $k$.

To show this, first note that there is $m \in \N$ such that $\frac{n}{n-2} < \lambda^m$.
Then
\[
\frac{n+\beta_j}{n-1+\beta_j} = \frac{\frac{n}{2}+\frac{n}{2}\lambda^j}{\frac{n}{2}-1+\frac{n}{2}\lambda^j} = \frac{\lambda^j+1}{\lambda^j+\frac{n-2}{n}} < \frac{\lambda^j+1}{\lambda^j+\lambda^{-m}}.
\]
Hence, for all $k \geq m$,
\[
\begin{split}
\prod_{j=0}^{k} \frac{n+\beta_j}{n-1+\beta_j} <& \prod_{j=0}^{k} \frac{\lambda^j+1}{\lambda^j+\lambda^{-m}}\\
=& \prod_{j=k-m+1}^{k}\frac{1}{\lambda^j + \lambda^{-m}} \prod_{j=0}^{k-m} \lambda^m \prod_{j=0}^{m-1} (\lambda^j+1)\\
\leq& \left( \prod_{j=k-m+1}^{k} \lambda^{-j} \right) \lambda^{m(k-m+1)} C(m,\lambda)\\
=& \lambda^{-(k-\frac{m-1}{2})m}\lambda^{m(k-m+1)} C(m,\lambda)\\
=& \lambda^{-\frac{m-1}{2}m} C(m,\lambda) = C_{17}.
\end{split}
\]

Hence, \eqref{Moser result} can be written by
\begin{equation} \label{Moser result2}
M_k \leq \left( C_{16}^{C_{17}} \right)^{\sum_{j=0}^{k} \frac{1}{n-1+\beta_j} } \cdot \left( \prod_{j=1}^{k} (1+\beta_j)^{\frac{1}{n-1+\beta_j}}\right)^{2C_{17}} M_0^{C_{17}}.
\end{equation}

Since $\beta_j = \frac{n}{2}\lambda^j - \frac{n}{2}$,
\[
\sum_{j=0}^{k} \frac{1}{n-1+\beta_j} = \sum_{j=0}^{k} \frac{1}{\frac{n}{2}-1+\frac{n}{2}\lambda^j} < \frac{2}{n} \sum_{j=0}^{k} \lambda^{-j} < \frac{2}{n} \frac{1}{1-\lambda^{-1}}.
\]

{\underline {Claim 2}:}
\[
\prod_{j=1}^{k} (n-1+\beta_j)^{\frac{1}{n-1+\beta_j}} \leq C_{18}
\]
where $C_{18}$ is independent on $k$.

The claim can be shown if we show $I(1) < \infty$ where
%
\begin{equation}
I(\alpha) = \int_{0}^{\infty} \frac{\ln (\frac{n}{2}-1+\frac{n}{2} \alpha \lambda^x)}{\frac{n}{2}-1+\frac{n}{2} \lambda^x} dx.
\end{equation}
Note that $I(0) < \infty$.
By direct computation, we have
\[
\begin{split}
I'(\alpha) =& \int_{0}^{\infty} \frac{\frac{n}{2} \lambda^x}{(\frac{n}{2}-1+\frac{n}{2}\lambda^x)(\frac{n}{2}-1+\frac{n}{2}\alpha \lambda^x)} dx\\
=& \int_{0}^{\infty} \left[ \frac{1}{\alpha-1} \frac{1}{\frac{n}{2}-1+\frac{n}{2}\lambda^x} - \frac{1}{\alpha-1} \frac{1}{\frac{n}{2}-1+\frac{n}{2} \alpha \lambda^x} \right] dx\\
=& \frac{2}{(\alpha-1)(n-2)} \left[ \int_{0}^{\infty} \frac{1}{1+\frac{n}{n-2} \lambda^x} dx - \int_{0}^{\infty} \frac{1}{1 + \frac{n \alpha }{n-2} \lambda^x} dx \right]\\
=& \frac{2}{(\alpha-1)(n-2)} \left[ \frac{1}{\ln \lambda} \ln \left( \frac{\frac{n}{n-2}+1}{\frac{n}{n-2}} \right) -  \frac{1}{\ln \lambda} \ln \left( \frac{\frac{n \alpha}{n-2}+1}{\frac{n \alpha}{n-2}} \right) \right]\\
=& \frac{2}{(\alpha-1)(n-2) \ln \lambda} \ln \left( \frac{(2n-2)\alpha}{n\alpha + n-2}\right).
\end{split}
\]
We want to estimate $\int_{\delta}^{1} I'(\alpha) d\alpha < C$ which is independent of $\delta>0$.

Fix $\delta_0>0$.
Then for any $0<\delta<\delta_0$,
\[
\begin{split}
\int_{\delta}^{\delta_0} I'(\alpha) d\alpha <& \frac{2}{(n-2)\ln \lambda} \int_{\delta}^{\delta_0} \left( \ln \left( \frac{n \delta_0 + n - 2}{2n-2} \right) - \ln \alpha \right) d\alpha\\
=& \frac{2}{(n-2)\ln \lambda} \left( \ln \left( \frac{n \delta_0 + n - 2}{2n-2} \right) (\delta_0-\delta) - [\alpha \ln \alpha - \alpha]_{\delta}^{\delta_0} \right)\\
\leq& C(\delta_0)
\end{split}
\]
because $\lim_{\delta \to 0} \delta \ln \delta = 0$.

Also, by $\ln(1+x) \geq \frac{x}{1+x}$ for all $x > -1$,
\[
\begin{split}
I'(\alpha) =& \frac{2}{(\alpha-1)(n-2) \ln \lambda} \ln \left( 1 + \frac{(n-2)(\alpha -1)}{n\alpha+n-2} \right)\\
\leq& \frac{2}{(n\alpha+n-2) \ln \lambda} \frac{1}{1 + \frac{(n-2)(\alpha -1)}{n\alpha+n-2}} = \frac{2}{\ln \lambda} \frac{1}{(2n-2)\alpha}.
\end{split}
\]
Hence,
\[
\begin{split}
\int_{\delta}^{1} I'(\alpha) d\alpha =& \int_{\delta}^{\delta_0} I'(\alpha) d\alpha + \int_{\delta_0}^{1} I'(\alpha) d\alpha\\
\leq& C(\delta_0) + \int_{\delta_0}^{1} \frac{2}{\ln \lambda (2n-2)\alpha} d\alpha\\
=& C(\delta_0) + \frac{-2 \ln \delta_0}{\ln \lambda(2n-2)}
\end{split}
\]
and by taking $\delta \to 0$, we obtain
\[
I(1) = I(0) + \int_{0}^{1} I'(\alpha) d\alpha < \infty.
\]
This completes proof of Claim 2.

Finally, from $n < \mu_0 = n \frac{n+\frac{1}{8}-1}{n+\frac{1}{8}-2} < 2n$, and by \eqref{Lq}, we have
\[
M_0 = \left( \sup_{[t_1',t_2]} \int_{B_{\rho_0}} e_2(f)^{\mu_0} +1 \right)^{\frac{1}{\mu_0}} < (C_{15} + 1)^{\frac{1}{n}}.
\]

Therefore, \eqref{Moser result2} becomes
\begin{equation}
M_k \leq \left(C_{16}^{C_{17}} \right)^{\frac{2}{n(1-\lambda^{-1})}} C_{18}^{2C_{17}} (C_{15}+1)^{\frac{C_{17}}{n}} =: C_{19}.
\end{equation}
By taking $k \to \infty$, we complete the proof.
\end{proof}

\section{Short time existence}
\label{sec7}

In this section we show short time existence using Banach fixed point theorem for $\ep>0$.
Note that our system \eqref{eq5} is a pair of evolution equations, and any standard existence theory may not be applied directly.
Also, uncoupling of the system using \eqref{e nu} is not a suitable method because the resulting evolution equation becomes
\[
f_t = \frac{e^{nat}\left( \Delta_n^{\ep} f + e_2(f)^{\frac{n}{2}-1} A(f)(df,df) \right)}{1 + nb \int_{0}^{t} e^{nas} e_2(f)^{\frac{n}{2}} (s) ds}
\]
which is non-local in time, hence very difficulty to analyze.

We mainly use similar notations in Mantegaza-Martinazzi \cite{MM12} and Park \cite{P22}.
Throughout this section and next section, we assume $\ep>0$ fixed.

In most of the arguments in this section, we consider $N \hookrightarrow \R^L$ isometrically.
Define spaces $P^m(M,T)$ be the completion of $C^\infty(M \times [0,T],\R^L)$ under the norm
\[
\|f\|_{P^m(M,T)}^2 := \sum_{2j + k \leq 2m} \int_{M \times [0,T]} |\partial_t^j \nabla^k f|^2 dx dt.
\]
Also define the space $Z^m(M,T)$ be the completion of $C^\infty(M \times [0,T])$ under the norm
\[
\|u\|_{Z^m(M,T)}^2 := \sum_{2j + k \leq 2m-1 }\int_{M \times [0,T]} |\partial_t^j \nabla^k u|^2 dx dt.
\]
If the operator $\partial_t - e^{-nu} \Delta_n^{\ep}$ is uniformly parabolic, then the map
\begin{equation}
\Phi : P^m(M,T) \to W^{2m-1,2}(M) \times P^{m-1}(M,T), \quad \Phi(f) = (f_0, (\partial_t - L)f)
\end{equation}
is a $C^1$ mapping for $m > 1 + \frac{n}{4}$. (Lemma 2.5 in \cite{MM12})

Throughout this section, we consider $f \in P^3(M,T)$ and $u \in Z^3(M,T)$.
Note that $m=3$ is needed to ensure that $m > 1 + \frac{n}{4}$.
By Sobolev embedding, there exists a universal constant $C$ such that
\[
\begin{split}
 \|\nabla^5 f\|_{L^{2+\frac{4}{n}}(M \times [0.T])} + \|\nabla^3 f_t\|_{L^{2+\frac{4}{n}}(M \times [0.T])} + \|\nabla f_{tt}\|_{L^{2+\frac{4}{n}}(M \times [0.T])} \leq& C \|f\|_{P^3},\\
 \|f\|_{C^0(M \times [0.T])} + \|\nabla f\|_{C^0(M \times [0.T])} + \|\nabla^2 f\|_{C^0(M \times [0.T])} + \|f_t\|_{C^0(M \times [0.T])} \leq& C \|f\|_{P^3},\\
\|\nabla^4 u\|_{L^{2+\frac{4}{n}}(M \times [0.T])} + \|\nabla^2 u_t\|_{L^{2+\frac{4}{n}}(M \times [0.T])} + \|u_{tt}\|_{L^{2+\frac{4}{n}}(M \times [0.T])} \leq& C \|u\|_{Z^3},\\
\|u\|_{C^0(M \times [0.T])} + \|\nabla u\|_{C^0(M \times [0.T])} \leq& C \|u\|_{Z^3}.
\end{split}
\]
Also, by standard parabolic theory, we have
\[
\begin{split}
\sup_{0 \leq t \leq T} \|f(t)\|_{W^{5,2}(M)} + \sup_{0 \leq t \leq T} \|f_t(t)\|_{W^{3,2}(M)} + \sup_{0 \leq t \leq T} \|f_{tt}(t)\|_{W^{1,2}(M)} \leq& C \|f\|_{P^3},\\
\sup_{0 \leq t \leq T} \|u(t)\|_{W^{4,2}(M)} + \sup_{0 \leq t \leq T} \|u_t(t)\|_{W^{2,2}(M)} + \sup_{0 \leq t \leq T} \|u_{tt}(t)\|_{L^{2}(M)} \leq& C \|u\|_{Z^3}.
\end{split}
\]
Note that for $n=3,4$, Sobolev embedding implies that for any $f \in W^{2,2}(M)$ and for any $q>1$,
\[
\|f\|_{L^q(M)} \leq C\|\nabla f\|_{L^4(M)} \leq C \|\nabla^2 f\|_{L^2(M)}.
\]

Fix $f_0 \in W^{5,2}(M,\R^L)$.
If $\|u\|_{C^0} \leq 1$, then the operator $\partial_t - e^{-nu}\Delta_n^{\ep}$ is uniformly parabolic.
Hence there exists $c_1>0$ such that for each $g \in P^2$, there exists a unique solution $h \in P^3$ of the equation
\begin{equation}
(\partial_t - e^{-nu} \Delta_n^{\ep}) h = g, h(0)=f_0
\end{equation}
with
\begin{equation} \label{c_1}
\|h\|_{P^3} \leq c_1 (\|f_0\|_{W^{5,2}} + \|g\|_{P^2}).
\end{equation}

Let $h_0$ be the unique solution of the equation
\[
(\partial_t - \Delta_n^{\ep}) h_0 = 0, h(0) = f_0.
\]
By \eqref{c_1}, there exists a constant $c_2$ only depending on $c_1$ and $\|f_0\|_{W^{5,2}}$ such that
\begin{equation}
\|h_0\|_{P^3} \leq c_2.
\end{equation}

Now for $\delta>0$, consider the closed ball $B_\delta$ in $P^3$
\begin{equation} \label{B_delta}
B_\delta = \{f \in P^3 : f(0) = f_0 \text{ and } \| f - h_0\|_{P^3} \leq \delta\}.
\end{equation}
Also, for $\delta' > 0$, consider the closed ball $\tilde{B}_{\delta'}$ in $Z^3$
\begin{equation} \label{B_delta '}
\tilde{B}_{\delta'} = \{u \in Z^3 : u(0) = 0 \text{ and } \|u \|_{Z^3} \leq \delta'\}.
\end{equation}
Here we assume $\delta'$ small enough so that $\|u\|_{Z^3} \leq \delta'$ implies $\|u\|_{C^0} \leq 1$.
We also assume $\delta ' \leq \frac{\delta}{2 c_1 c_3}$ where $c_3$ is given in \Cref{necessary}.

We are going to construct two operators $S_1 : P^3 \times Z^3 \to P^3$ and $S_2 : P^3 \times Z^3 \to Z^3$ as follows.
Define $S_1(f,u) = h$ where $h \in P^3$ is the unique solution of
\begin{equation} \label{S_1}
(\partial_t - \Delta_n^{\ep}) h = e^{-nu} e_2(f)^{\frac{n}{2}-1} A(f)(df,df), h(0) = f_0.
\end{equation}
Similarly, define $S_2(f,u) = v$ where $v \in Z^3$ is the unique solution of
\begin{equation} \label{S_2}
\partial_t v = b e^{-nu} e_2(f)^{\frac{n}{2}} - a, v(0)=0.
\end{equation}

\begin{prop} \label{B to B}
For $f \in P^3$, $u \in Z^3$, $S_1(f,u) \in P^3$ and $S_2(f,u) \in Z^3$.

Moreover, there exists $T_1 = T_1(c_2,\delta,\delta')>0$ such that $S_1$ restricts to $S_1 : B_\delta \times \tilde{B}_{\delta'} \to B_\delta$ and $S_2$ restricts to $S_2 : B_\delta \times \tilde{B}_{\delta'} \to \tilde{B}_{\delta'}$.
\end{prop}

The following lemma is needed in the proof.

\begin{lemma} \label{necessary}
There exists $T_2 = T_2(c_2,\delta,\delta')>0$ such that for all $T \leq T_2$, for any $h \in B_\delta$ and for each $u_1,u_2 \in \tilde{B}_{\delta'}$,
\begin{equation}
\|(e^{nu_2-nu_1}-1) \partial_t h\|_{P^2} \leq c_3 \|u_2-u_1\|_{Z^3}
\end{equation}
for some constant $c_3$ only depending on $c_2$.
\end{lemma}

\begin{proof}
Observe that $|e^{nu_2-nu_1}-1 | \leq C |u_2-u_1|$ for some constant $C$ if $\|u_1\|_{C^0},\|u_2\|_{C^0} \leq 1$.
Also, for $u \in Z^3$, with $\|u\|_{C^0} \leq 2$, we have
\[
\begin{split}
|\nabla^4 (e^u)| \leq& e^u |\nabla^4 u| + C \left( |\nabla^3 u| |\nabla u| + |\nabla^2 u|^2 + |\nabla^2 u| |\nabla u|^2 + |\nabla u|^4 \right)
\end{split}
\]
hence for all $T$ small enough, $\| \nabla^4 (e^u)\|_{L^2(M \times [0,T])} \leq C \|\nabla^4 u\|_{L^2(M \times [0,T])}$.
Similar argument gives that for all $T$ small enough, we have $\|\partial_t^j \nabla^k (e^u)\|_{L^2(M \times [0,T])} \leq C(\|u\|_{C^0}) \|\partial_t^j \nabla^k u \|_{L^2(M \times [0,T])}$ for all $T$ small enough, with $2j+k \leq 4$.

Consider
\[
\begin{split}
|\nabla^4 \left( (e^{nu_2-nu_1}-1) \partial_t h \right) | \leq& C \left( |\nabla^4 e^{n(u_2-u_1)}| |\partial_t h| \right.\\
& \left. + |\nabla^3 e^{n(u_2-u_1)}| |\partial_t \nabla h|  + |\nabla^2 e^{n(u_2-u_1)}| |\partial_t \nabla^2 h| \right.\\
& \left. + |\nabla e^{n(u_2-u_1)}| |\partial_t \nabla^3 h| + |u_2-u_1| |\partial_t \nabla^4 h|  \right).
\end{split}
\]
Hence, we get
\[
\begin{split}
\| \nabla^4 \left( (e^{nu_2-nu_1}-1) \partial_t h \right) \|_{L^2(M \times [0,T])}^2 \leq& C(\|h\|_{P^3})  \|u_2-u_1\|_{Z^3}^2 T\\
& +  C \|u_2-u_1\|_{C^0}^2 \|h\|_{P^3}^2\\
\leq&  C c_2 \|u_2-u_1\|_{Z^3}^2 
\end{split}
\]
if $T$ is small enough.
Similar argument will complete the proof.
\end{proof}

\begin{proof}
(Proof of \Cref{B to B})
First, $S_1(f,u) \in P^3$ is clear because $e^{-nu} e_2(f)^{\frac{n}{2}-1} A(f)(df,df) \in P^2$.
And $v = S_2(f,u) \in Z^3$ comes from the fact that
\begin{equation} \label{v est}
\begin{split}
\| \nabla^5 v \|_{L^2(M \times [0,T])}^2 \leq& C(\|f\|_{P^3}, \|u\|_{Z^3}) \sum_{k_1+k_2=5} \| | \nabla^{k_1+1} f|^2 |\nabla^{k_2} u|^2 \|_{L^2(M \times [0,T])}^2 T^2\\
\leq& C(\|f\|_{P^3}, \|u\|_{Z^3}) T^2\\
\|\partial_t \nabla^3 v\|_{L^2(M \times [0,T])}^2 \leq& C(\|f\|_{P^3}, \|u\|_{Z^3}) \sum_{k_1+k_2 =3} \| |\nabla^{k_1+1} f|^2 |\nabla^{k_2} u|^2 \|_{L^2(M \times [0,T])}^2\\
\leq& C(\|f\|_{P^3}, \|u\|_{Z^3}) T\\
\|\partial_t^2 \nabla v\|_{L^2(M \times [0,T])}^2 \leq& C(\|f\|_{P^3}, \|u\|_{Z^3}) \left( \| \partial_t \nabla^2 f \|_{L^2(M \times [0,T])}^2 + \| \partial_t \nabla u \|_{L^2(M \times [0,T])}^2 \right.\\
&  \quad + \left. \| \partial_t \nabla f  \|_{L^2(M \times [0,T])}^2 + \| \partial_t  u \|_{L^2(M \times [0,T])}^2 \right)\\
\leq& C(\|f\|_{P^3}, \|u\|_{Z^3}) T.
\end{split}
\end{equation}
This observation also suggests that, for $f \in B_\delta$ and $u \in \tilde{B}_{\delta'}$, and for $T$ small enough, we have $v = S_2(f,u) \in \tilde{B}_{\delta'}$.

Next, denote $h = S_1(f,u)$ and $g = e^{-nu} e_2(f)^{\frac{n}{2}-1} A(f)(df,df)$.
$h-h_0$ satisfies
\[
(\partial_t - \Delta_n^{\ep}) (h - h_0) = g + (e^{-nu}-1) \partial_t h, (h-h_0)(0) = 0.
\]
Hence by \eqref{c_1}, we get
\[
\|h-h_0\|_{P^3} \leq c_1 \|g + (e^{-nu}-1) \partial_t h\|_{P^2} \leq c_1 \|g\|_{P^2} + c_1 \|(e^{-nu}-1) \partial_t h\|_{P^2}.
\]

We show that for all $T$ small enough,
\[
\| \nabla^4 g\|_{L^2(M \times [0,T])}, \|\partial_t \nabla^2 g\|_{L^2(M \times [0,T])}, \|\partial_t^2 g\|_{L^2(M \times [0,T])} \leq \frac{\delta}{6c_1}.
\]
Note that using suitable Sobolev embedding and $\|\nabla f\|_{C^0} \leq \|f\|_{P^3}, \|\nabla u\|_{C^0} \leq \|u\|_{Z^3}$, we get
\[
\begin{split}
|\nabla^4 g | \leq& C(\|f\|_{P^3}, \|u\|_{Z^3}) \left( |\nabla^4 u| + |\nabla^3 u| |\nabla^2 f| + |\nabla^2 u| |\nabla^3 f| + |\nabla^2 u| |\nabla^2 f|^2 \right.\\
& \left. + |\nabla u|^2 |\nabla^3 f| + |\nabla u| |\nabla^4 f| + |\nabla u| |\nabla^3 f| |\nabla^2 f| + |\nabla u| |\nabla^2 f|^3 + |\nabla^5 f| \right)\\
\| \nabla^4 g\|_{L^2(M \times [0,T])}^2 \leq& C(\|f\|_{P^3}, \|u\|_{Z^3}) \left( \sup_{0 \leq t \leq T} \|f(t)\|_{W^{5,2}}^2 + \sup_{0 \leq t \leq T} \|u(t)\|_{W^{4,2}}^2 \right) T \leq \left( \frac{\delta}{6c_1} \right)^2
\end{split}
\]
if $T$ is small enough.
Similarly, $\|\partial_t \nabla^2 g\|_{L^2(M \times [0,T])}, \|\partial_t^2 g\|_{L^2(M \times [0,T])}$ can be controlled if $T$ is small enough.

For the second term, from \Cref{necessary},
\[
\|(e^{-nu}-1) \partial_t h\|_{P^2} \leq c_3 \|u\|_{Z^3} \leq \frac{\delta}{2c_1}.
\]
This completes the proof.
\end{proof}

Next, we show that difference in inputs of $S_1,S_2$ can be controlled by the corresponding norms in $P^3$ or $Z^3$.
More precisely, we show the following proposition.

\begin{prop} \label{S est}
There exists $T_3 = T_3(c_2,\delta,\delta')>0$ such that for all $T \leq T_0$ the followings hold.
\begin{enumerate}
\item
For any $f \in B_\delta$ and $u_1,u_2 \in \tilde{B}_{\delta'}$, we have
\begin{align}
\|S_1(f,u_1)-S_1(f,u_2)\|_{P^3} \leq& c_1 c_3 \|u_1-u_2\|_{Z^3} \label{est 1}\\
\|S_2(f,u_1)-S_2(f,u_2)\|_{Z^3} \leq& \frac{1}{3} \|u_1-u_2\|_{Z^3}. \label{est 2}
\end{align}
\item
For any $f_1,f_2 \in B_\delta$ and $u \in \tilde{B}_{\delta'}$, we have
\begin{align}
\|S_1(f_1,u)-S_1(f_2,u)\|_{P^3} \leq& \frac{1}{3} \|f_1-f_2\|_{P^3} \label{est 3}\\
\|S_2(f_1,u)-S_2(f_2,u)\|_{Z^3} \leq& T^{\frac{1}{4}} \|f_1-f_2\|_{P^3}. \label{est 4}
\end{align}
\end{enumerate}
\end{prop}

\begin{proof}

{\underline {Case (1)}:}

Fix $f \in B_\delta$ and $u_1,u_2 \in \tilde{B}_{\delta'}$.
Denote $h_1 = S_1(f,u_1), h_2 = S_1(f,u_2), v_1 = S_2(f,u_1), v_2 = S_2(f,u_2)$.
Then $h_1-h_2$ satisfies
\[
(\partial_t - e^{-nu_1} \Delta_n^{\ep}) (h_1-h_2) = (e^{nu_2-nu_1}-1) \partial_t h_2, (h_2-h_1)(0) = 0.
\]
Then by \eqref{c_1},
\[
\|h_1-h_2\|_{P^3} \leq c_1 \|(e^{nu_2-nu_1}-1) \partial_t h_2\| \leq c_1 c_3 \|u_2-u_1\|_{Z^3}
\]
which shows \eqref{est 1}.

$v_1-v_2$ satisfies
\[
\begin{split}
v_1-v_2 =& b \int_{0}^{T} e_2(f)^{\frac{n}{2}} (e^{-n(u_1-u_2)}).
\end{split}
\]
Hence, similar to \eqref{v est},
\[
\begin{split}
\|\nabla^5 (v_1-v_2)\|_{L^2(M \times [0,T])}^2 \leq& C( \delta, \delta') \sum_{k_1+k_2 = 5} \| |\nabla^{k_1+1} f|^2 |\nabla^{k_2} (u_1-u_2)|^2 \|_{L^2(M \times [0,T])}^2 T^2\\
\leq& C(\delta, \delta')\|f\|_{P^3}^2 \|u_1-u_2\|_{Z^3}^2 T^2\\
\|\partial_t \nabla^3 (v_1-v_2)\|_{L^2(M \times [0,T])}^2 \leq& C( \delta, \delta') \sum_{k_1+k_2 = 3} \| |\nabla^{k_1+1} f|^2 |\nabla^{k_2} (u_1-u_2)|^2 \|_{L^2(M \times [0,T])}^2 \\
\leq& C(\delta, \delta')  \|f\|_{P^3}^2 \|u_1-u_2\|_{Z^3}^2 T\\
\|\partial_t^2 \nabla (v_1-v_2)\|_{L^2(M \times [0,T])}^2 \leq& C( \delta, \delta') \left( \|\partial_t \nabla^2 f\|_{L^2(M \times [0,T])}^2 \|u_1-u_2\|_{C^0(M \times [0,T])}^2 \right.\\
& + \|f\|_{C^0(M \times [0,T])}^2 \|\partial_t \nabla (u_1-u_2) \|_{L^2(M \times [0,T])}^2\\
& + \|\partial_t \nabla f\|_{L^2(M \times [0,T])}^2 \|\nabla (u_1-u_2) \|_{C^0(M \times [0,T])}^2\\
& \left. + \|\nabla^2 f\|_{C^0(M \times [0,T])}^2 \|\partial_t (u_1-u_2)\|_{L^2(M \times [0,T])}^2 \right)\\
\leq& C(\delta, \delta') \|f\|_{P^3}^2 \|u_1-u_2\|_{Z^3}^2 T
\end{split}
\]
which implies that for all $T$ small enough, \eqref{est 2} holds.

{\underline {Case (2)}:}

Fix $f_1,f_2 \in B_\delta$ and $u \in \tilde{B}_{\delta'}$.
Denote $h_1 = S_1(f_1,u), h_2 = S_1(f_2,u), v_1 = S_2(f_1,u), v_2 = S_2(f_2,u)$.
Then $h_1-h_2$ satisfies
\[
(\partial_t - e^{-nu}\Delta_n^{\ep}) (h_1-h_2) = e^{-nu} \left( e_2(f_1)^{\frac{n}{2}-1} A(f_1)(df_1,df_1) - e_2(f_2)^{\frac{n}{2}-1} A(f_2)(df_2,df_2)  \right)
\]
and $(h_1-h_2)(0) = 0$.
Now the right-hand side is
\[
e^{-nu} \left( e_2(f_1)^{\frac{n}{2}-1} A(f_1)(df_1,df_1) - e_2(f_2)^{\frac{n}{2}-1} A(f_2)(df_2,df_2)  \right) = I + II + III
\]
where
\[
\begin{split}
I =& e^{-nu} \left(e_2(f_1)^{\frac{n}{2}-1} - e_2(f_2)^{\frac{n}{2}-1} \right) A(f_1) (df_1,df_1)\\
II =& e^{-nu} e_2(f_2)^{\frac{n}{2}-1} (A(f_1)-A(f_2)) (df_1,df_1)\\
III =& e^{-nu} e_2(f_2)^{\frac{n}{2}-1} A(f_2) (df_1+df_2, df_1-df_2).
\end{split}
\]
We claim that $\|I\|_{P^2}, \|II\|_{P^2}, \|III\|_{P^2} \leq \frac{1}{9 c_1} \|f_1-f_2\|_{P^3}$ if $T$ is small enough.
To see this, note that $|e_2(f_1)^{\frac{n}{2}-1} - e_2(f_2)^{\frac{n}{2}-1}| \leq C(\delta) |\nabla (f_1-f_2)|$ and that
\[
\begin{split}
|\nabla^4 I| \leq& C(\delta, \delta') \left( |\nabla^4 u| |\nabla (f_1-f_2)| + |\nabla^3 u| |\nabla^2 (f_1-f_2)| + |\nabla^2 u| |\nabla^3 (f_1-f_2)|  \right.\\
& \left. + |\nabla^2 u| |\nabla^2 (f_1-f_2)|^2 + |\nabla u| |\nabla^4 (f_1-f_2)| + |\nabla u| |\nabla^3 (f_1-f_2)| |\nabla^2 (f_1-f_2)| \right.\\
& \left. + |\nabla u| |\nabla^2 (f_1-f_2)|^3 + |\nabla^5(f_1-f_2)| + |\nabla^4 (f_1-f_2)| |\nabla^2 (f_1-f_2)|  \right.\\
& \left. + |\nabla^3 (f_1-f_2)|^2 + |\nabla^3 (f_1-f_2)| |\nabla^2 (f_1-f_2)|^2 + |\nabla^2 (f_1-f_2)|^4 \right)\\
\|\nabla^4 I\|_{L^2(M \times [0,T])}^2 \leq& C(\delta,\delta') \sup_{0 \leq t \leq T} \|f_1-f_2\|_{W^{5,2}}^2 T\\
\leq& \left( \frac{1}{36 c_1}\right)^2 \|f_1-f_2\|_{P^3}^2
\end{split}
\]
if $T$ is small enough.
Similar arguments can conclude the claim, and hence complete \eqref{est 3}.

$v_1-v_2$ satisfies
\[
v_1 - v_2 = b \int_{0}^{T} e^{-nu} \left(e_2(f_1)^{\frac{n}{2}-1} - e_2(f_2)^{\frac{n}{2}-1} \right).
\]
Hence, similar to \eqref{v est},
\[
\begin{split}
\|\nabla^5 (v_1-v_2)\|_{L^2(M \times [0,T])}^2 \leq& C( \delta, \delta') \sum_{k_1+k_2 = 5} \| |\nabla^{k_1+1} (f_1-f_2)|^2 |\nabla^{k_2} u|^2 \|_{L^2(M \times [0,T])}^2 T^2\\
\leq& C(\delta, \delta')\|f_1-f_2\|_{P^3}^2 \|u\|_{Z^3}^2 T^2\\
\|\partial_t \nabla^3 (v_1-v_2)\|_{L^2(M \times [0,T])}^2 \leq& C( \delta, \delta') \sum_{k_1+k_2 = 3} \| |\nabla^{k_1+1} (f_1-f_2)|^2 |\nabla^{k_2} u|^2 \|_{L^2(M \times [0,T])}^2 \\
\leq& C(\delta, \delta')  \|f_1-f_2\|_{P^3}^2 \|u\|_{Z^3}^2 T\\
\|\partial_t^2 \nabla (v_1-v_2)\|_{L^2(M \times [0,T])}^2 \leq& C( \delta, \delta') \left( \|\partial_t \nabla^2 (f_1-f_2)\|_{L^2(M \times [0,T])}^2 \|u\|_{C^0(M \times [0,T])}^2 \right.\\
& + \|f_1-f_2\|_{C^0(M \times [0,T])}^2 \|\partial_t \nabla u \|_{L^2(M \times [0,T])}^2\\
& + \|\partial_t \nabla (f_1-f_2)\|_{L^2(M \times [0,T])}^2 \|\nabla u \|_{C^0(M \times [0,T])}^2\\
& \left. + \|\nabla^2 (f_1-f_2)\|_{C^0(M \times [0,T])}^2 \|\partial_t u\|_{L^2(M \times [0,T])}^2 \right)\\
\leq& C(\delta, \delta') \|f_1-f_2\|_{P^3}^2 \|u\|_{Z^3}^2 T
\end{split}
\]
which implies that for all $T$ small enough, \eqref{est 4} holds.
\end{proof}

Now we are ready to show short-time existence.
First, we define a Banach space $X = P^3 \times Z^3$ equipped with the norm
\begin{equation}
\|(f,u)\|_{X} = (2c_1c_3)^{-1} \|f\|_{P^3} + \|u\|_{Z^3}.
\end{equation}

\begin{theorem} \label{short time existence}
Fix $f_0 \in W^{5,2}(M,N)$.
Then there exists $T_0 = T_0(c_2,\delta,\delta')>0$ such that a solution $(f,u) \in B_\delta \times \tilde{B}_{\delta'} \subset P^3 \times Z^3$ of \eqref{eq5} exists on $M \times [0,T_0]$.
\end{theorem}

\begin{proof}
We let $T_0 = T_0(c_1,c_2,c_3, \delta,\delta')$ by
\begin{equation}
T_0 = \min \left\{ T_1,T_2,T_3, \left( \frac{1}{2} (2 c_1 c_3)^{-1} \right)^4 \right\} > 0.
\end{equation}
Now consider the operator $\mathcal{S} : X \to X$ given by $\mathcal{S} (f,u) = (S_1(f,u),S_2(f,u))$.
From \Cref{B to B}, since $T_0 \leq T_1$, $\mathcal{S}$ restricts to $\mathcal{S} : B_\delta \times \tilde{B}_{\delta'} \to B_\delta \times \tilde{B}_{\delta'}$.

Let $f_1,f_2 \in B_\delta$ and $u_1,u_2 \in \tilde{B}_{\delta'}$.
From \Cref{S est}, for any $T \leq T_0$, we have
\[
\begin{split}
\|\mathcal{S}(f_1,u_1)& - \mathcal{S}(f_2,u_2)\|_{X}\\
 =& (2c_1 c_3)^{-1} \|S_1(f_1,u_1) - S_1(f_2,u_2)\|_{P^3} + \|S_2(f_1,u_1) - S_2(f_2,u_2)\|_{Z^3}\\
\leq&(2c_1 c_3)^{-1} \|S_1(f_1,u_1) - S_1(f_1,u_2)\|_{P^3} + (2c_1 c_3)^{-1} \|S_1(f_1,u_2) - S_1(f_2,u_2)\|_{P^3}\\
& + \|S_2(f_1,u_1) - S_2(f_1,u_2)\|_{Z^3} + \|S_2(f_1,u_2) - S_2(f_2,u_2)\|_{Z^3}\\
\leq& \frac{1}{2} \|u_1 -u_2\|_{Z^3} + (2c_1 c_3)^{-1}  \frac{1}{3} \|f_1-f_2\|_{P^3} + \frac{1}{3} \|u_1-u_2\|_{Z^3} + T^{\frac{1}{4}} \|f_1-f_2\|_{P^3}\\
\leq& \frac{5}{6} \Big( (2c_1 c_3)^{-1} \|f_1-f_2\|_{P^3} + \|u_1-u_2\|_{Z^3} \Big)\\
 =& \frac{5}{6} \|(f_1,u_2)-(f_2,u_2)\|_{X}.
\end{split}
\]
By Banach fixed point theorem, there exists $(f,u) \in P^3 \times Z^3$ such that $S_1(f,u) = f$ and $S_2(f,u)=u$.
This shows that $(f,u)$ solves \eqref{eq5}.

Finally, we need to show that $f(M \times [0,T_0]) \subset N$.
As $N$ be a smooth Riemannian manifold isometrically embedded in $\R^L$, we can find its tubular neighborhood $N_\delta$ and consider the nearest point projection $\Pi_N : N_\delta \to N$.
Note that $\nabla \Pi(y) : \R^L \to T_y N$ is the orthogonal projection and $A(y) = -\nabla^2 \Pi(y) : T_y N \otimes T_y N \to (T_y N)^{\perp}$ for $y \in N$ is the second fundamental form of $N \hookrightarrow \R^L$.
Define
\begin{equation}
\rho(x,t) = \left| \Pi(f(x,t)) - f(x,t) \right|^2.
\end{equation}
Then $\rho(\cdot,0) = 0$ and by direct computation, we have
\[
\begin{split}
\partial_t \rho - e^{-nu} \Div (e_2(f)^{\frac{n}{2}-1} \nabla \rho) = & 2 \langle \Pi(f)-f, \nabla \Pi (f_t) - f_t \rangle\\
& - 2 e^{-nu} e_2(f)^{\frac{n}{2}-1} \left| \nabla (\Pi(f)-f) \right|^2\\
&- 2 e^{-nu} \langle \Pi(f)-f, -\Delta_n^{\ep}(f) - e_2(f)^{\frac{n}{2}-1} A(f)(df,df) \rangle\\
& + 2 e^{-nu} \langle \Pi(f)-f, e_2(f)^{\frac{n}{2}-2} \langle \nabla df, df \rangle \nabla \Pi(\nabla f) \rangle\\
=&- 2 e^{-nu} e_2(f)^{\frac{n}{2}-1} \left| \nabla (\Pi(f)-f) \right|^2 \leq 0.
\end{split}
\]
So by maximum principle, $\rho \equiv 0$ and so $f(M \times [0,T_0]) \subset N$.
\end{proof}

\section{Global regularity}
\label{sec8}

In this section we develop global regularity for $\ep>0$.
From \Cref{sec6}, there exists $T_0>0$ such that the smooth solution $(f,u)$ of \eqref{eq5} exists on $M \times [0,T_0)$.
Assume $T_0$ be the maximal time such that the solution $(f,u)$ is smooth on $M \times [0,T_0)$.
Also assume that $T_0 < \infty$, that is, finite time singularity exists.
We first show that the criterion for finite time singularity is energy concentration.
Then we show that such energy concentration cannot be obtained, concluding that there is no finite time singularity.

First we show that the solution $(f,u)$ of \eqref{eq5} obtained in \Cref{short time existence} is smooth.
\begin{prop} \label{smooth}
Let $(f,u) \in P^3 \times Z^3$ be a solution of \eqref{eq5} on $M \times [0,T]$.
Then $(f,u)$ is smooth on $M \times [0,T]$.
\end{prop}

\begin{proof}
Since $f \in P^3$, we have $\|df\|_{C^0} \leq C \|f\|_{P^3} =: M < \infty$.
Then
\[
e^{-nu} = \frac{e^{nat}}{1 + nb \int_{0}^{t} e^{nas} e_2(f)^{\frac{n}{2}} (s) ds} \geq \frac{1}{1 + \frac{b}{a} {M}^{n}}.
\]
So, the operator $\partial_t - e^{-nu} \Delta_n^{\ep}$ is uniformly parabolic, hence by bootstrapping argument with $df \in C^{\alpha,\alpha/2}$, $f$ is smooth, hence $u$ is also smooth on $M \times [0,T)$.
\end{proof}

Next, we develope the global version of \Cref{Der p=0}.
\begin{prop} \label{Der p=0 glob}
(Derivative estimate for $p=0$, global version)
Let $(f,u)$ be a smooth solution of \eqref{eq5} on $M \times [t_1,t_2]$.
Then
\begin{equation}
\begin{split}
\frac{d}{dt} \frac{1}{2} \int_{M} e^{nu}|f_t|^2  \leq& \frac{na}{2} \int_{M} e^{nu}|f_t|^2 \\
&-\frac{1}{4} \int_{M} e_2(f)^{\frac{n}{2}-1} |d f_t|^2 \\
&-\frac{n-2}{2} \int_{M} e_2(f)^{\frac{n}{2}-2} (\langle df, df_t \rangle)^2\\
&+ \left( C_N + 2C_N^2 - \frac{nb}{2} \right) \int_{M} e_2(f)^{\frac{n}{2}} |f_t|^2 .
\end{split}
\end{equation}
\end{prop}

Its proof is almost the same as \Cref{Der p=0}.
As a result of \Cref{Der p=0 glob} and \Cref{E dec}, we have
\[
\int_{M} e^{nu}|f_t|^2  (t_2) - \int_{M} e^{nu}|f_t|^2 (t_1) \leq na \int_{t_1}^{t_2} \int_{M} e^{nu}|f_t|^2 = na (E^{\ep}(t_1) - E^{\ep}(t_2))
\]
hence for any $t \geq 0$,
\begin{equation} \label{int f_t^2 bound}
\int_{M} e^{nu}|f_t|^2 (t) \leq na E^{\ep}(0).
\end{equation}

\begin{theorem} \label{finitely many}
Let $(f,u)$ be a smooth solution of \eqref{eq5} on $M \times [0,T_0)$.
Assume $T_0<\infty$ is the maximal existence time.
Then there exists at most finitely many points $x_1, \cdots, x_k$ such that
\begin{equation} \label{finite time sing criterion}
\lim_{r \to 0} \limsup_{t \nearrow T_0} \int_{B_r(x_k)} e_2(f)^{\frac{n}{2}} > \bar{\ep}.
\end{equation}
\end{theorem}

\begin{proof}
We first show that if for $x \in M$,
\begin{equation} \label{small E}
\lim_{r \to 0} \limsup_{t \nearrow T_0} \int_{B_r(x)} e_2(f)^{\frac{n}{2}} \leq \bar{\ep},
\end{equation}
then $f$ and $u$ is smooth at $(x,T_0)$.

From above assumption, we may assume that for some fixed $t_1$ with $0 < t_1 < T_0$, for all $r>0$ small enough, $\sup_{[t_1,T_0]} \int_{B_r(x)} e_2(f)^{\frac{n}{2}} (t) \leq \bar{\ep}$.
Since $f$ is smooth at $t_1$, we choose $0 < r_1 \leq r$ such that $\int_{B_{r_1}(x)} e^{nu}|f_t|^2 (t_1) \leq \bar{\ep}$.
Also, since $u$ is smooth at $t_1' = \frac{t_1+T_0}{2}$, let $C_u < \infty$ be a constant such that $\sup_{x \in B_{r_1}} u(x,t_1') \leq C_u$.
Then by \Cref{L infty}, $e_2(f) \leq C_{19}$ on $B_{\frac{r_1}{2}} \times [t_1',T_0]$ and hence by similar argument to \Cref{smooth},
we obtain that $f$ is smooth and $u$ is also smooth at $(x,T_0)$.

Next, assume that for all $x \in M$, \eqref{small E} holds.
Then $f,u$ are smooth on $M \times \{T_0\}$, hence by \Cref{short time existence}, the solution $(t,u)$ exists on $M \times [0,T_0+\delta]$ for some $\delta>0$.
This conflicts with the assumption that $T_0$ is the maximal existence time.
Therefore, there should be $x \in M$ such that for any $r>0$,
\begin{equation} \label{large E}
\limsup_{t \nearrow T_0} \int_{B_r(x)} e_2(f)^{\frac{n}{2}} > \bar{\ep}.
\end{equation}

We show that there are at most finitely many such points.
Let $x_1, \cdots, x_k$ be any finite collection of such points.
Fix $R>0$ so that $B_{2R}(x_i)$ are disjoint.
Let $T_0' = T_0 - \frac{\bar{\ep}}{4 C_3 \left( 1 + \frac{1}{R^n} \right)^2}$.
Choose $t_i$ such that
\[
 \int_{B_R(x_i)} e_2(f)^{\frac{n}{2}} (t_i) > \frac{\bar{\ep}}{2}.
\]
From \Cref{loc E under small} on $B_{2R}(x_i)$, we have
\[
\begin{split}
\frac{\bar{\ep}}{2} <& \int_{B_R(x_i)} e_2(f)^{\frac{n}{2}} (t_i) \leq \int_{B_{2R}(x_i)} e_2(f)^{\frac{n}{2}} (T_0') + C_2 \left( \int_{B_{2R}(x_i)} e_2(f)^{\frac{n}{2}} (T_0') + \int_{B_{2R}(x_i)} e^{nu}|f_t|^2 (T_0') \right)\\
& + C_3 \left( 1 + \frac{1}{R^n} \right)^2 (T_0-T_0')\\
\frac{\bar{\ep}}{4} <& (1 + C_2) \int_{B_{2R}(x_i)} e_2(f)^{\frac{n}{2}} (T_0') + C_2 \int_{B_{2R}(x_i)} e^{nu}|f_t|^2 (T_0')\\
k \frac{\bar{\ep}}{4} <& (1 + C_2) \int_{M} e_2(f)^{\frac{n}{2}} (T_0') + C_2 \int_{M} e^{nu}|f_t|^2 (T_0'). 
\end{split}
\]
From \Cref{E dec}, $\int_{M} e_2(f)^{\frac{n}{2}} (T_0') \leq E^{\ep}(0)$.
Also, from \eqref{int f_t^2 bound}, $\int_{M} e^{nu}|f_t|^2 (T_0') \leq na E^{\ep}(0)$.
This completes the proof.
\end{proof}

Next, we show that actually there is no such finite time singularity.

\begin{proof}
(Proof of \Cref{main 1})
From \Cref{sec6}, there exists $T_0>0$ such that the smooth solution $(f,u)$ of \eqref{eq5} exists on $M \times [0,T_0)$.
Assume $(x_0,T)$ be a finite time singularity and fix $B_R(x_0)$ be a ball centered at $x_0$.
The condition \eqref{finite time sing criterion} implies
\[
\limsup_{t \nearrow T} \int_{B_R(x_0)} e_2(f)^{\frac{n}{2}} (t) = K + \int_{B_R(x_0)} e_2(f)^{\frac{n}{2}} (T)
\]
for some positive constant $K > \bar{\ep} > 0$, measuring the energy loss at $(x_0,T)$.
Equivalently, $K$ can be described by
\[
K = \lim_{r \to 0} \limsup_{t \nearrow T} \int_{B_r(x_0)} e_2(f)^{\frac{n}{2}} (t).
\]

Let $0 < r \leq R$ and $\varphi$ be a cut-off function on $B_r(x_0)$.
Define the local energy
\begin{equation}
\Theta_r(t) = \int_{B_r(x_0)} e_2(f)^{\frac{n}{2}} \varphi^n.
\end{equation}
Then
\[
\begin{split}
\frac{\Theta_r(t)}{dt} =& \frac{n}{2} \int_{B_r(x_0)} \varphi^n e_2(f)^{\frac{n}{2}-1} \langle df_t, df \rangle\\
=& -\frac{n}{2} \int_{B_r(x_0)} \varphi^n \langle f_t, \Delta_n^{\ep} f \rangle - \frac{n^2}{2} \int_{B_r(x_0)} \varphi^{n-1} \langle f_t, \nabla \varphi \cdot e_2(f)^{\frac{n}{2}-1} df \rangle\\
\leq & -\frac{n}{2} \int_{B_r(x_0)} \varphi^n e^{nu} |f_t|^2 + \frac{C}{r} \int_{B_r(x_0)} |f_t| e_2(f)^{\frac{n}{2}-1} |df|\\
\leq& \frac{C}{r} \left( \int_{B_r(x_0)}  e_2(f)^{\frac{n}{2}} |f_t|^2 \right)^{\frac{1}{2}} \left( \int_{B_r(x_0)} e_2(f)^{\frac{n}{2} -1} \right)^{\frac{1}{2}}\\
\leq& C \left( \int_{B_r(x_0)}  e_2(f)^{\frac{n}{2}} |f_t|^2 \right)^{\frac{1}{2}}.
\end{split}
\]
So, for $s \leq t < T$, by \Cref{n2 est},
\begin{equation}
\begin{split}
\Theta_r(t) - \Theta_r(s) =& C \int_{s}^{t} \left( \int_{B_r(x_0)} e_2(f)^{\frac{n}{2}} |f_t|^2 \right)^{\frac{1}{2}}\\
\leq& C (t-s)^{\frac{1}{2}} \left( \int_{s}^{t} \int_{B_R(x_0)} e_2(f)^{\frac{n}{2}} |f_t|^2 \right)^{\frac{1}{2}}\\
\leq& C_{20} (t-s)^{\frac{1}{2}}
\end{split}
\end{equation}
where the constant $C_{20}$ only depends on $E^{\ep}(0), C_b, T, R$.
Hence we can take the limit $\lim_{t \nearrow T} \Theta_r(t)$ and have
\[
K = \lim_{t \nearrow T} \int_{B_r(x_0)} e_2(f)^{\frac{n}{2}} (t) - \int_{B_r(x_0)} e_2(f)^{\frac{n}{2}} (T) = \lim_{t \nearrow T} \Theta_r (t) - \Theta_r(T).
\]
Combining above two inequality gives
\begin{equation} \label{Theta est}
|K + \Theta_r(T) - \Theta_r(s)| \leq C_{20} (T-s)^{\frac{1}{2}}.
\end{equation}
Now, fix $\delta >0$ such that
\begin{equation}
\delta = \min\{ \frac{1}{C_{20}}, \frac{K}{4}\}.
\end{equation}
Also fix $s = T - \delta^4$ and choose $r$ small enough so that $\Theta_r(T) < \frac{K}{4}$ and $\Theta_r(s) < \frac{K}{4}$.
Then from \eqref{Theta est} we have
\[
\frac{K}{2} \leq C_{20} (T-s)^{\frac{1}{2}} \leq C_{20} \delta^2 < \delta \leq \frac{K}{4}
\]
which is a contradiction.
Hence there is no finite time singularity.
Then by bootstrapping argument for uniformly parabolic equation, we obtain smooth solution $(f,u)$ of \eqref{eq5} and we complete the proof of \Cref{main 1}.
\end{proof}




\bibliographystyle{abbrv}
\bibliography{bib}

\end{document}